\documentclass{amsart}

\usepackage{mathtools,  stmaryrd}
\usepackage{xparse} \DeclarePairedDelimiterX{\Iintv}[1]{\llbracket}{\rrbracket}{\iintvargs{#1}}
\NewDocumentCommand{\iintvargs}{>{\SplitArgument{1}{,}}m}
{\iintvargsaux#1} %
\NewDocumentCommand{\iintvargsaux}{mm} {#1\mkern1.5mu,\mkern1.5mu#2}
\usepackage[T1]{fontenc}

\usepackage{graphicx}
\usepackage{xcolor}

\newcommand{\lf}{\lfloor}
\newcommand{\rf}{\rfloor}
\newcommand{\R}{\mathbb{R}}

\newcommand{\N}{\mathbb{N}}
\newcommand{\e}{\varepsilon}
\newcommand{\la}{\langle}
  \newcommand{\ra}{\rangle}
\newcommand{\E}{\mathbb{E}}
\newcommand{\Z}{\mathbb{Z}}

\renewcommand{\S}{\mathbb{S}}
\renewcommand{\P}{\mathbb{P}}
\renewcommand{\L}{\mathbb{L}}

\newcommand{\rmT}{\mathrm{T}}

\def\ben#1{\begin{equation}#1\end{equation}}

\def\al#1{\begin{align*}#1\end{align*}}
\def\aln#1{\begin{align}#1\end{align}}

\newtheorem{lem}{Lemma}[section]
\newtheorem{remark}[lem]{Remark}

\newtheorem{thm}[lem]{Theorem}

\newtheorem{cor}[lem]{Corollary}

\newtheorem{assum} {Assumption}
\title[A variational formula for large deviations in FPP]
{A variational formula for large deviations in First-passage percolation under tail estimates}

\date{\today}
\author{Cl\'ement Cosco} 
\address[Cl\'ement Cosco]
{Weizmann Institute of Science, Israel.}
\email{clement.cosco@gmail.com}

\author{Shuta Nakajima} 
\address[Shuta Nakajima]
{University of Basel, Basel, Switzerland}
\email{shuta.nakajima@unibas.ch}


\begin{document}
\maketitle

\begin{abstract}

Consider the first passage percolation on the $d$-dimensional lattice $\mathbb{Z}^d$ with identical and independent weight distributions and the first passage time ${\rm T}$. In this paper, we study the upper tail large deviations  $\mathbb{P}({\rm T}(0,nx)>n(\mu+\xi))$, for  $\xi>0$ and $x\neq 0$   with a time constant $\mu$, for weights that satisfy a tail assumption $ \mathbb{P}(\tau_e>t)\asymp \beta\exp{(-\alpha t^r)}.$ When $r\leq 1$ (this includes the well-known Eden growth model), we show that the upper tail large deviation decays as $\exp{(-(2d\alpha\xi^r +o(1))n)}$. 
When $1< r\leq d$, we find that the rate function can be naturally described by a variational formula, called the discrete p-Capacity,  and we study its asymptotics. {The case $r=d$ is critical and logarithmic corrections appear.} For $r\in (1,d)$, we show that the large deviation event $\{{\rm T}(0,nx)>n(\mu+\xi)\}$ is described by a localization of high weights around the endpoints. The picture changes for $r\geq d$ where the configuration is not anymore localized. 
\end{abstract}

\section{Introduction and Main results}
\subsection{Introduction}

First-passage percolation (FPP) was first introduced by Hammersley and Welsh in 1965 \cite{HW65}, as a dynamical version of the percolation model. Since then, it has been extensively studied both in mathematics and physics. There are several important and interesting aspects of FPP. First, FPP naturally defines a random metric space. Indeed, the objects of interest called the first passage time and the optimal path, correspond to a metric and a geodesic in a certain random metric space. Second, FPP is expected to belong to the KPZ universality class. Moreover, it is widely believed that the boundary of a ball defined by the first passage time behaves like a KPZ equation \cite{KPZ86,KS91}. In those studies, many important methods and phenomena have been discovered. See \cite{ADH} for more detailed backgrounds.\\

In this paper, we consider the first passage percolation on the lattice $\L^d=(\Z^d,E(\Z^d))$ with $d\geq 2$. Given nearest-neighbor vertices $v,w$, we write $\la v,w \ra{=\{v,w\}}$ the undirected edge connecting $v$ and $w$. The model is defined as follows. To each edge $e\in E(\Z^d)$, we assign a non-negative random variable $\tau_e$. We assume that the collection $\tau=(\tau_e)_{e\in E(\Z^d)}$ is identically and independently  distributed.

A sequence $(x_i)_{i=1}^{l}$ is said to be a path if each successive pair is nearest neighbor, i.e. $|x_i-x_{i+1}|_1=1$ for any $i$. We note that a path is seen both as a set of vertices and a set of edges with some abuse of notation. Given a path $\gamma$, we define the passage time of $\gamma$ as
$${\rm T}(\gamma)=\sum_{e\in\gamma}\tau_e.$$
For $x\in\R^d$, we set $\lf x\rf=(\lf x_1\rf,\cdots,\lf x_d\rf)$ where $\lf a\rf$ is the greatest integer less than or equal to $a$ for $a\in\R$. Given $x,y\in\R^d$, we define the {\em first passage time} between $x$ and $y$ as
$${\rm T}(x,y)=\inf_{\gamma:\lf x\rf\to \lf y\rf}{\rm T}(\gamma),$$
where the infimum is taken over all finite paths $\gamma$ starting at $\lf x\rf$ and ending at $\lf y\rf$. A path attaining the infimum is called an optimal path.

A striking feature of the first passage time is the sub-additivity, namely, for any $x,y,z\in\Z^d$,
\ben{\label{subadd}
\rmT(x,z)\leq \rmT(x,y)+\rmT(y,z).
}
Indeed, the right-hand side is the infimum of the first passage times over the set of all finite paths from $\lf x\rf$ to $\lf z\rf$ passing through $\lf y\rf$, which is a subset of the set of all finite paths from $\lf x\rf$ to $\lf z\rf$. This yields \eqref{subadd}. In particular, $(\Z^d,\rmT)$ is a pseudometric, and if, in addition, $\P(\tau_e=0)=0$, then this is exactly a metric.

Using the sub-additivity, Kingman \cite{King73} proved that if $\E \tau_e<\infty$, then for any $x\in\R^d$, there exists a non-random constant $\mu(x)\in [0,\E\tau_e]$ (called the {\em time constant}) such that

\begin{align}\label{kingman}
\mu(x)=\lim_{t\to\infty}t^{-1} {\rm T}(0,t x)=\lim_{t\to\infty}t^{-1} \E[{\rm T}(0,t x)]\hspace{4mm}\text{a.s.}
\end{align}

In this paper, we are interested in the upper tail large deviations:
\begin{equation} \label{eq:ULD}
\P({\rm T}(0,nx)>n(\mu(x)+\xi)),\quad \xi >0.
\end{equation}
It has been observed in \cite{CGM09} that the properties of \eqref{eq:ULD} may depend on the tails of the distributions of $\tau_e$. In this paper, we consider the following condition:


\begin{assum} \label{assumA}
there exist $\beta_1,\beta_2,\alpha>0$ and $r\in(0,\infty)$ such that for $t\geq 0$,
\begin{equation}\label{cond:distr}
\beta_1\exp{(-\alpha t^r)}\leq \P(\tau_e>t)\leq \beta_2\exp{(-\alpha t^r)}.
\end{equation}
\end{assum}
Note that when $\tau_e$ follows an exponential distribution, the model is also called the Eden growth model \cite{Eden}, in which case $B_t=\{x\in\R^d:{\rmT(0,x)\leq t} \}$ is a Markov process. For a generalization of the condition \eqref{cond:distr}, see Remark~\ref{remark:1} and Remark~\ref{remark:2} (see also \cite{CGM09}).

When $r\in(d,\infty)$ with $d \geq 2$ or $r=d=2$, 
the authors in \cite{CGM09} have given upper and lower bounds on \eqref{eq:ULD}.
In the present work, we give estimates for general $r>0$ and $d\geq 2$, and we obtain matching upper and lower bounds when $r\in (0,d]$. In the latter case, we find that the rate function is a power of $\xi$, times a constant coming from a variational formula.

\subsection{Main results}
\begin{thm}\label{thm:main1}
Suppose Assumption \ref{assumA} with $r\leq1$. Then for all $\xi>0$ and $x\in\R^d\backslash\{0\}$,
\ben{
\lim_{n\to\infty}\frac{1}{n^r}\log{\P({\rm T}(0,nx)>n(\mu(x)+\xi))}=-2d\alpha \xi^r,\label{thm:maineq1}
}
where $\alpha$ and $r$ are in \eqref{cond:distr}.
\end{thm}

\begin{remark}\label{remark:1}
For $r<1$, we can weaken Assumption \ref{assumA} as follows. Suppose that there exist slowly varying functions $\beta_1(t),\beta_2(t),\alpha_1(t),\alpha_2(t)$ and $0<r<1$ such that $\displaystyle\lim_{t\to\infty}\frac{\alpha_1(t)}{\alpha_2(t)}=1$ and for $t\geq 0$,
\ben{\label{cond-distr2}
\beta_1(t)\exp{(-\alpha_1(t)t^r)}\leq \P(\tau_e>t)\leq \beta_2(t)\exp{(-\alpha_2(t)t^r)},
}
where a measurable function $f(t):[0,\infty)\to(0,\infty)$ is said to be slowly varying if for any $a>0$,
$$\lim_{t\to\infty}\frac{f(at)}{f(t)}=1.$$
Then, \eqref{thm:maineq1} is replaced by the following:
\begin{equation}\label{thm:main2}
\lim_{n\to\infty}\frac{1}{\alpha_1(n)n^r}\log{\P({\rm T}(0,nx)>n(\mu(x)+\xi))}=-2d \xi^r.
\end{equation}
When $r=1$, if, in addition, $\alpha_1(t)$ is bounded away from $0$ and $\infty$, then the same holds. See Section~\ref{slow-vary} for the proof.
\end{remark}
\begin{remark}
Let us take $r=1$ and slowly varying functions $\alpha(t),\beta(t)$ such that $t \longrightarrow \alpha(t) t,\,\beta(t)^{-1}$ are strictly increasing and
$$0<\alpha_1:=\varliminf_t \alpha(t)<\alpha_2:=\varlimsup_t \alpha(t)<\infty.$$
Let us consider the distribution satisfying 
$$\P(\tau_e>t)=\beta(t)\exp{(-\alpha(t)t)}.$$
By the previous remark,
\aln{\label{thm:main3}
-2d \xi \alpha_2 &= \varliminf_{n\to\infty}\frac{1}{n}\log{\P({\rm T}(0,nx)>n(\mu(x)+\xi))}\\
&<\varlimsup_{n\to\infty}\frac{1}{n}\log{\P({\rm T}(0,nx)>n(\mu(x)+\xi))}=-2d \xi \alpha_1.
}
Hence, in general, the extra $\alpha_1(n)$ scaling factor in \eqref{thm:main2} cannot be omitted.
\end{remark}

We turn to the case $r>1$. Let
\begin{equation*}
D_M(x)=\{y\in \mathbb Z^d: |y-x|_\infty \leq M\},
\end{equation*} be the closed ball in $\mathbb Z^d$ of radius $M$. Let
\begin{equation*}
\partial D_M(x) =\{y\in \mathbb Z^d: |y-x|_\infty = M\},
\end{equation*}
denote its boundary and 
\begin{equation*}
E_{M}=\{\la x,y\ra :x,y\in D_M(0),\,|x-y|_1=1\},
\end{equation*}
be the set of edges included in $D_M(0)$.  We define
\begin{equation*}
\begin{aligned}
&\lambda_{d,r}(M)= \\
&
\inf_{f: D_M(0)\to \R}\left\{\sum_{\langle x, y\rangle \in E_M} |f(x)-f(y)|^{r}\ \middle|~\forall x\in \partial D_M(0),\,f(x)\geq 1,\,f(0)=0\right\}.
\end{aligned}
\end{equation*}
Since $M\to\lambda_{d,r}(M)$ is non-increasing, we can consider
\begin{equation}
\label{eq:def_lambdaLimit}\lambda_{d,r} = \lim_{M\to\infty} \lambda_{d,r}(M).
\end{equation}

\begin{remark}
This $\lambda_{d,r}$ is closely connected to the notion of p-Capacity \cite[Chapter 2]{Flucher}. Hence, we call it the discrete r-Capacity in this article. Note that when $r=2$, $\lambda_{d,2}=1-\pi_d$, where $\pi_d$ is the return probability of the simple random walk of dimension $d$, see \cite[pp. 25 \& 36]{LP16}.
\end{remark}

\begin{thm} \label{th:rLessThandBig1} Suppose Assumption \ref{assumA} with $1<r<d$. For all $\xi>0$ and $x\in\R^d\backslash\{0\}$,
\begin{equation}\label{maineq2}
\lim_{n\to\infty} \frac{1}{n^r} \log \P\left({\rm T}(0,nx) \geq (\mu(x)+\xi) n \right)= -\alpha 2^{1-r} \lambda_{d,r} \,\xi^r.
\end{equation}
\end{thm}

Before turning to the case $r=d$, we introduce the notation:
\al{
\kappa_{d,d}:={\rm vol}_{d-1}\left(\left\{x\in \R^d:~|x|_{\frac{d}{d-1}} =1\right\}\right),
}
where ${\rm vol}_{d-1}$ is the $d-1$ dimensional volume and $|x|_{s}=(\sum_{i=1}^d |x_i|^{s})^{\frac{1}{s}}$.

\begin{thm} \label{th:requalsd} Suppose Assumption \ref{assumA} with $r=d$. For all $\xi>0$ and $x\in\R^d\backslash\{0\}$, there exists $M>0$ such that
\begin{equation}
\begin{aligned}
\lim_{n\to\infty} \frac{1}{n^d\,\lambda_{d,d}(n)} \log \P\big({\rm T}(0,nx)\geq (\mu+\xi) n \big)
= -\alpha 2^{1-d} \, \xi^d.\label{main thm for d=r}
\end{aligned}
\end{equation}
\end{thm}

In the critical case $r=d$, one can obtain the following asymptotic behavior of the $p$-capacity:
\begin{thm}[\cite{CN21}] \label{th:upcoming}
For all $d\geq 2$,
$
\lim_{n\to\infty}\, (\log n)^{d-1}\lambda_{d,d}(n)=\kappa_{d,d}.
$
\end{thm}
\begin{remark} \label{rem:wellknown}
The result is well-known when $d=2$, see e.g.\ \cite[Proposition 6.3.2]{L10} and \cite[pp. 25 \& 36]{LP16}.
\end{remark}

Combined with Theorem \ref{th:requalsd}, 
we obtain:

\begin{cor} \label{cor:main}
Suppose Assumption \ref{assumA} with $r=d\geq 2$. For all $\xi>0$ and $x\in\R^d\backslash\{0\}$,
\al{
\lim_{n\to\infty} \frac{(\log{n})^{d-1}}{n^d} \log \P\left({\rm T}(0,nx)\geq (\mu+\xi) n \right)
= -\alpha 2^{1-d} \, \xi^d\,\kappa_{d,d}.
}
\end{cor}


\begin{remark}\label{remark:2}
As in Remark~\ref{remark:1}, for $r\in [1,d)$, we can weaken Assumption \ref{assumA} as follows. Suppose that there exist slowly varying functions $\beta_1(t),\beta_2(t),\alpha_1(t),\alpha_2(t)$ such that $\displaystyle\lim_{t\to\infty}\frac{\alpha_1(t)}{\alpha_2(t)}=1$ and \eqref{cond-distr2} holds. Then, \eqref{maineq2} is replaced by the following:
\ben{\label{thm:main4}
\lim_{n\to\infty}\frac{1}{\alpha_1(n)n^r}\log{\P({\rm T}(0,nx)>n(\mu(x)+\xi))}=-2^{1-r} \lambda_{d,r}\, \xi^r.
}
See Section~\ref{slow-vary2} for the proof.

When $r=d$, we can weaken Assumption \ref{assumA} as
$$\lim_{t\to\infty} \frac{1}{t^r}\log{\P(\tau_e>t)}=\alpha.$$
The proof is essentially the same as before and we leave it for the readers.  We remark that we never expect \eqref{thm:main4} with general slowly varying functions when $r=d$, since edge weights around the endpoints are not of order $n$ (see Theorem~\ref{th:rLessThandBig1Loc}) and thus $\alpha_1(n)$ may not appear in the rate function.
\end{remark}

More generally, we can obtain large deviation estimates on the whole scale $N\geq n$ for first passage times on the 
positive quadrant of $\mathbb Z^d$ which is used to prove the previous theorems: 
\begin{thm}\label{prop1}
Let $d\geq{}1$ and $r>1$. Suppose that there exist constants $\alpha,\beta>0$ such that for $t\geq 0$,
\begin{equation}\label{UB condition}
\begin{split}
\P(\tau_e\geq{}t)\le \beta e^{-\alpha t^{r}}.
\end{split}
\end{equation}
Then, for all $\xi>0$, there exists $c=c(\xi,r,d)>0$ such that for all $n\in \mathbb N$ and all $N\in[n,\infty)$,
\begin{equation} \label{eq:thInduction}
\P({\rm T}_{[0,n]^d}\left(0,{n}\mathbf{e}_1)> (\mu(\mathbf{e}_1)+\xi) N\right) \le \exp{\left(-c\,{\rm g}_{d,r}(n,N)\right)},
\end{equation}
where  ${\rm T}_{[0,n]^d}$ is the restricted first passage time on $[0,n]^d$ defined in Section~\ref{notation and terminology}, and 
\begin{equation*}
\begin{split}
{\rm g}_{d,r}(n,N)=\begin{cases}
N^r&\text{ if $1<r<d$,}\\
N^d/(1+\log{n})^{d-1}&\text{ if $r=d$,}\\
N^r/n^{r-d}&\text{ if $r>d$.}\\
\end{cases}
\end{split}
\end{equation*}
\end{thm}

When $r<d$, {the large deviation event results from the weights around the origin taking high values}, while for $r\geq d$, the contribution is spread out at higher range. What we can show is the following transition in the behavior between $r<d$ and $r\geq d$:
\begin{thm} \label{th:rLessThandBig1Loc}
Suppose Assumption \ref{assumA} and let $x\in\R^d\backslash\{0\}$. If $1<r<d$, then
there exists $R>0$ such that
for all $\xi>0$, there is $\varepsilon_0=\varepsilon_0(\xi)>0$ such that
\begin{equation} \label{eq:loc}
\P\left(\forall e\in E_R, \tau_e\leq \varepsilon_0 n \, \middle| \,\mathrm T(0,nx) \geq (\mu(x)+\xi) n \right) \to 0,
\end{equation}
as $n\to\infty$.

On the other hand, if $r\geq d$, then for all $a>0$, $\xi>0$ and $\varepsilon_0>0$,
\begin{equation} \label{eq:Noloc}
\P\left(\forall e\in E_{n^a}, \tau_e\leq \varepsilon_0 n \, \middle| \,\mathrm T(0,nx) \geq (\mu(x)+\xi) n \right) \to 1.
\end{equation}
\end{thm}

We end this section by considering the asymptotics of the rate function near the origin. Under the condition \eqref{UB condition} with $r>d$, it is expected that the following quantity admits a finite and positive limit:
\begin{equation} \label{eq:r>dNolimit}
-\frac{1}{n^d}\P({\rm T}(0,n\mathbf{e}_1)>(\mu(\mathbf{e}_1)+\xi)n).
\end{equation}
See Section~\ref{section related work} for related works. The authors in \cite{CGM09} proved that the limsup and liminf of \eqref{eq:r>dNolimit} are both finite and positive. We only consider the limsup: given $\xi>0$, we define
$$\overline{I}(\xi)=\varlimsup_{n\to\infty}\left(-\frac{1}{n^d}\P({\rm T}(0,n\mathbf{e}_1)>(\mu(\mathbf{e}_1)+\xi)n)\right)\in (0,\infty).$$
\begin{thm}\label{asymptotic of the rate function}
Assume the condition \eqref{UB condition} with some $r>d$ and $\tau_e$ is unbounded, i.e. for any $m>0$, $\P(\tau_e>m)>0$. Then
$$\varlimsup_{\xi\to 0+} \xi^{-d} \overline{I}(\xi)<\infty.$$
\end{thm}
\begin{remark}We are not sure whether the limit above is positive or not.
\end{remark}

\subsection{Related works and discussion}\label{section related work}
Large deviations principle is one of the major subjects in probability theory. The study of the large deviations in the {context of} first passage percolation was initiated by Kesten \cite{Kes86}. For the lower tail {large deviations}, by using the usual subadditivity argument, he obtained that for $\xi>0$ small enough, the following limit exists and is negative:
\ben{\label{Kes: lower}
\lim_{n\to\infty}\frac{1}{n}\log{ \P({\rm T}(0,n\mathbf{e}_1)<n(\mu(\mathbf{e}_1)-\xi))}.
}
On the other hand, he showed that under the boundedness of the distribution,
\aln{
-\infty&<\varliminf_{n\to\infty}\frac{1}{n^d}\log{ \P({\rm T}(0,n\mathbf{e}_1)>n(\mu(\mathbf{e}_1)+\xi))}\nonumber\\
&\leq \varlimsup_{n\to\infty}\frac{1}{n^d}\log{ \P({\rm T}(0,n\mathbf{e}_1)>n(\mu(\mathbf{e}_1)+\xi))}<0. \label{Kes: upper}
}
It is worth noting that the rates of upper {large deviations} and lower {large deviations} are different (See \cite{CZ03} for the heuristics of the difference).  It is natural to expect the two limits in \eqref{Kes: upper} coincide, in which case we call the limit the rate function. The authors in \cite{CZ03} confirm the existence of the rate function of upper tail large deviations for the line-to-line first passage time. Although that for the point-to-point case as in \eqref{Kes: upper} had been an open problem for many years, it was solved in \cite{BGS} recently for bounded distributions with continuity densities. The assumption of boundedness is essential since the rate $n^d$ in \eqref{Kes: upper} may change for unbounded distribution, as illustrated in the previous section and first observed in \cite{CGM09}, {see also \cite{A15} for distributions with power law decay}. Let us now give a quick comparison of our results and those of \cite{CGM09}.

In \cite{CGM09}, the authors have proven that under $\log{\P(\tau_e > x)} = -x^{d} f(x)$ where $f$ is a positive increasing function, \eqref{Kes: upper} holds if and only if $$\sum_{n\in \N} f(2^n)^{-1/{(d-1)}}<\infty.$$ In particular, this implies that \eqref{Kes: upper} holds when $\tau_e$ satisfies Assumption \ref{assumA} with $r>d$. Our Theorem \ref{prop1} is thus a slight refinement of their result. When $d=2$ and $\tau_e$ follows the Half-normal distribution, they also proved that
\begin{equation} \label{eq:boundsr=d}
\begin{aligned}
-\infty & < \varlimsup_{n\to\infty} \frac{(\log{n})^{d-1}}{n^d} \log \P\left({\rm T}(0,nx)\geq (\mu+\xi) n \right) \\
& \leq \varlimsup_{n\to\infty} \frac{(\log{n})^{d-1}}{n^d} \log \P\left({\rm T}(0,nx)\geq (\mu+\xi) n \right)
< 0.
\end{aligned}
\end{equation}

In this paper, we extend \eqref{eq:boundsr=d} for any distribution satisfying Assumption \ref{assumA} with $r=d\geq 2$. Moreover, we find matching upper and lower bounds when $d=2$ and exhibit the rate function (Corollary \ref{cor:main}). A forthcoming paper will give matching bounds for $d\geq 3$. Another new contribution of the paper is the case $r<d$ (Theorem~\ref{thm:main1} and Theorem \ref{th:rLessThandBig1}), which has not been touched in \cite{CGM09}. In this case, the scaling is changed by $n^r$, and we obtain upper and lower bounds that match and exhibit the rate function. In Theorem \ref{prop1}, we further obtain an upper bound for different scales of deviations and spatially constrained paths, for any $r> 1$ and $d\geq 1$. Our approach is quite different from \cite{CGM09}. Here, thanks to our estimates on $\lambda_{d,r}$, we can follow a quite short and straightforward induction argument.

The different scalings in Theorem~\ref{prop1} come from different scenarios of the upper {large deviations} event $\{{\rm T}(0,nx)>(\mu(x)+\xi)n\}$. Indeed, for light-tailed distributions, the upper {large deviations} events are affected by overall configurations. In contrast, for heavy-tailed distributions, the upper {large deviations} events highly depend on the configurations around the endpoints. Hence, we need to take a distribution-dependent approach to study upper tail {large deviations} in-depth for general distribution. This paper considers Weibull distributions with a shape parameter less than or equal to the dimension. In these cases, by studying the neighborhoods of the endpoints carefully, we can do a more detailed analysis, which enables us to get the exact value of the rate function, which is characterized by the discrete p-Capacity. To analyze the asymptotics of the discrete p-Capacity, we propose a new method (to our knowledge), that combines a pathwise formulation of the problem (Section~\ref{section: path formalization}) with linear constraints, and a 
reduction to a unique constraint via the introduction of a random path, allowing us to apply the Lagrangian multiplier method (Section~\ref{section: sketch}). Although the route we take is quite different, the latter argument is inspired by the ideas of \cite{CGM09}.

Incidentally, upper tail large deviations with scaling $n^r$, $r\leq 1$, appear in several models, such as Last-passage percolation and Directed polymers with stretched exponential distributions \cite{BA09}, the chemical distance of percolation clusters \cite{GM07}, and the frog models \cite{AMP02}. For Last-passage percolation models, our methods are applicable and one can get the analogous results. In the frog models, it is proved that the object
$$-\log{\P({\rm T}_{\rm f}(0,nx)>(\mu_{\rm f}(x)+\xi)n)},$$ with the first passage time ${\rm T}_{\rm f}$ for the frog model and a certain time constant $\mu_{\rm f}(x)$, grows like (i) $\sqrt{n}$ for $d=1$; (ii) $n/\log n$ for $d=2$; (iii) $n$ for $d\geq 3$ \cite{CKN}. Moreover, new variational formulas for the rate functions are given there as in our results. In other models, we expect similar results.

Finally, we describe a relation between the upper tail large deviations and the maximal edge-traversal time, $\max_{e\in \gamma} \tau_e$, of optimal paths. Indeed, it was revealed in \cite{N17} that the growth order of the maximal edge-traversal time is more or less determined by the scaling order of the upper tail large deviations. More precisely, under the Assumption \ref{assumA} with some additional assumptions, we have for any optimal path $\gamma$ from $0$ to $n\mathbf{e}_1$,
$$\max_{e\in \gamma} \tau_e \asymp
\begin{cases}
(\log{n})^{\frac{1}{1+r}} & \text{ if }r<d-1,\\
(\log{n})^{\frac{1}{d}} & \text{ if }d-1<r<d,\\
(\log{n})^{\frac{1}{r}} & \text{ if }r>d,
\end{cases}
$$
where the transition at $r=d-1$ above  corresponds to the transition at $r=d$ in upper tail {large deviations}; the transition at $r=d$ above comes from a different reason. 

%
%
%

\subsection{Notation and terminology}\label{notation and terminology}
This subsection collects some useful notations, terminologies and remarks.
\begin{itemize}
\item {For any set $A$, we denote by $|A|$ its cardinality.}
\item {For $D\subset \mathbb R^d$ and $\gamma = (x_i)_{i=0}^l$, we write $\gamma \subset D$ when $x_i \in D$ for all $i \in [1,\ell-1]$. If $E\subset  E(\Z^d)$, then we also write $\gamma\subset E$ when $\la x_i,x_{i+1}\ra \in E$ for all $i$.}
 \item {Given two vertices $v,w\in\Z^d$ and a set $D\subset\R^d$, we set the {\em restricted} first passage time as
$${\rm T}_D(v,w)=\inf_{\gamma\subset D}{\rm T}(\gamma),$$
   where the infimum is taken over all paths $\gamma$ from $v$ to $w$ with $\gamma\subset D$. If such a path does not exist, then we set it to be infinity instead.}
 \item Similarly, given a set $E\subset  E(\Z^d)$, we define
   $${\rm T}_{E}(v,w)=\inf_{\gamma\subset E}{\rm T}(\gamma).$$
   
 \item Given two sets $A\subset \mathbb R^d$ (or $A\subset  E(\Z^d)$) and $B\subset \mathbb Z^d$, and a point $v\in\mathbb Z^d$, we write
 \[{\rm T}_A(v,B)=\inf_{w\in B} {\rm T}_A(v,w).\]

  \item Given $(x_i)_{i=1}^k\subset \R$ and random variables $(X_i)_{i=1}^k$ with $k\in\N$,
   \al{\P\left(\sum_{i=1}^k X_i\geq \sum_{i=1}^k x_i\right)&\leq \P\left(\exists i\in\{1,\cdots,k\}\text{ s.t. } X_i\geq  x_i\right)\\
   &\leq \sum_{i=1}^k\P(X_i\geq x_i).
 }
   We use this inequality throughout this paper without any comment.
\end{itemize}
\section{Preliminaries on $\lambda_{d,r}(n)$}
\subsection{Pathwise formulation for $\lambda_{d,r}(n)$}\label{section: path formalization}
 We consider a generalization of $\lambda_{d,r}(M)$ {(which is defined above \eqref{eq:def_lambdaLimit})}: given a connected subset $A$ of $\Z^d$ containing $0$ and a set $B$ with $B\subset A$, we let
 \ben{\label{general lambda}
   \lambda_{d,r}(A;B)=\inf_{f \in \R^A}\left\{\sum_{\langle x, y\rangle \in E_A} |f(x)-f(y)|^{r}\ \middle|~\forall x\in B,\,f(x)\geq 1,\,f(0)=0\right\},
   }
where $E_A=\{e = \langle x ,  y \rangle | x,y\in A\}$. {For any nearest neighbor finite path $\gamma$ in $\mathbb Z^d$ and any set $A\subset \mathbb Z^d$, we use the notation $\gamma:0\to A$ when the path $\gamma$ starts at $0$ and ends in $A$,} and we define
\begin{equation} \label{def:lambda}
\lambda_{d,r}^{\rm P}(A;B)= \inf_{(t_e)\in (\mathbb R_+)^{E_A}} \left\{\sum_{e\in E_A} t_e^{r}\ \middle|~\forall \gamma:0\to B\text{ with $\gamma\subset A$},\,\sum_{e\in\gamma}t_e\geq 1\right\}.
\end{equation}
\begin{thm}\label{thm: path formalization}
  For any $B\subset A\subset \Z^d$ and $r>0$,
  \ben{\label{path-potential duality}
    \lambda_{d,r}(A;B)=\lambda_{d,r}^{\rm P}(A;B).
      }
In particular,
\begin{equation} \label{def:lambdaIntro}
\lambda_{d,r}(n)= \lambda^P_{d,r}(D_n(0);\partial D_n(0)).
\end{equation}

\end{thm}
\begin{proof}
  We first prove $\lambda_{d,r}(A;B)\leq \lambda^{\rm P}_{d,r}(A;B)$. Given $(t_e)\in (\mathbb R_+)^{E_A}$, we define the function $f:A\to \R$ as 
    $$f(x)=\inf_{\gamma} \sum_{e\in\gamma} t_e,\,x\in A,$$
  where the infimum is taken over all paths $\gamma$ from $0$ to  $x$ inside $A$. For $e=\langle x,y\rangle\in E_A$, by the triangular inequality, $f(y)\leq f(x)+ t_e$ and $f(x)\leq f(y)+ t_e$, which implies $ |f(y)-f(x)|\leq t_e$ and $$\sum_{\langle x, y\rangle \in E_A}|f(x)-f(y)|^{r}\leq \sum_{e\in E_A} t_e^{r}.$$ Moreover, if for any path $\gamma$ from $0$  to  $B$  with $\gamma\subset A$, we have $\sum_{e\in\gamma}t_e\geq 1$, then  $f(0)=0$ and for any $x\in B,\,f(x)\geq 1$. Therefore, we get the desired inequality.

    Next, we prove $\lambda_{d,r}(A;B)\geq \lambda^{\rm P}_{d,r}(A;B)$. To this end, we take a function $f: A\to \R$ such that $\forall x\in B,\,f(x)\geq 1,\,f(0)=0$. Let us define $t_e=|f(x)-f(y)|$ for $e=\langle x,y\rangle$. Then for any path $\gamma=(x_0,\cdots,x_l)$ from $0$ to $B$ with $\gamma\subset A$,
      $$\sum_{e\in\gamma} t_e = \sum_{i=1}^l |f(x_i)-f(x_{i-1})| \geq f(x_{l})-f(0)=f(x_l)\geq 1,$$
  which completes the proof.
\end{proof}
\subsection{Asymptotics of $\lambda_{d,r}(n)$}
Let us define
$$
\kappa_{d,r}(n)=
\begin{cases}
  \lambda_{d,r}(n)&\text{ if }r<d,\\
 (\log{n})^{d-1} \lambda_{d,r}(n)&\text{ if }r=d,\\
  n^{r-d} \lambda_{d,r}(n)&\text{ if }r>d.
\end{cases}$$
\begin{thm}\label{asymptotic for lambda general}
  For $d\geq 1$ and $r>0$, $\kappa_{d,r}(n)$ is bounded away from $0$ and $\infty$.
\end{thm}
  We give a proof in Section~\ref{section: sketch}. 

\section{The lower bounds}
For the simplicity of notation, we only consider the case $x=\mathbf{e}_1$, though essentially the same proof works for general $x$. For the proof of general $x$, see Remark~\ref{rem: general x}. We write 
\[{\rm T}_n={\rm T}(0,n\mathbf{e}_1) \quad \text{and} \quad \mu=\mu(\mathbf{e}_1). 
\]

Recall that $D_M(x)=\{y\in \mathbb Z^d: |y-x|_\infty \leq M\}$. 
Throughout the paper, we denote by $C_d>0$ some constant such that for all $M>0$, $|D_M(0)| \leq C_d M^d$ and $|\partial D_M(x)| \leq C_d M^{d-1}$. 

\subsection{Proof of the lower bound in Theorem \ref{thm:main1} ($r\leq 1$)}

We fix $\xi>0$ and take $\e\in(0,\xi)$ arbitrary. We denote
$$\tilde{E}_{1}=\{e\in E(\Z^d):~0\in e\}\text{ and } \tilde{D}_1(0)=\{x\in\Z^d:~|x|_1=1\}.$$

If we suppose
$$\forall e\in  \tilde{E}_{1},\,\tau_e>(\xi+\e)n\text{ and }\min_{x\in \tilde{D}_1(0)}{\rm T}_{ E(\Z^d)\backslash  \tilde{E}_{1}}(x,n\mathbf{e}_1)>(\mu-\e)n,$$
 since
$${\rm T}_n=\min_{x\in \tilde{D}_{1}}(\tau_{\langle 0,x\rangle}+{\rm T}(x,n\mathbf{e}_1))\geq \min_{e\in \tilde{E}_{1}}\tau_e+\min_{x\in \tilde{D}_1(0)}{\rm T}_{ E(\Z^d)\backslash  \tilde{E}_{1}}(x,n\mathbf{e}_1), $$
then  we get ${\rm T}_n>(\mu+\xi)n$. Thus
\aln{
 & \qquad \P({\rm T}_n>(\mu+\xi)n)\nonumber\\
  &\geq \P(\forall e\in \tilde{E}_{1},~\tau_e>(\xi+\e)n,\,\min_{x\in \tilde{D}_1(0)}{\rm T}_{ E(\Z^d)\backslash  \tilde{E}_{1}}(x,n\mathbf{e}_1)>(\mu-\e)n)\nonumber\\
  &= \P\left(\forall e\in  \tilde{E}_{1},~\tau_e>(\xi+\e)n\right)\,\P\left(\min_{x\in \tilde{D}_1(0)}{\rm T}_{ E(\Z^d)\backslash  \tilde{E}_{1}}(x,n\mathbf{e}_1)>(\mu-\e)n\right).\label{lower-est}
}
The first term can be bounded from below by  $\beta_1^{2d}\exp{(-2d\alpha (\xi+\e)^r n^r)}$, where $\beta_1$ is in \eqref{cond:distr}. On the other hand, for the second term, since $${\rm T}_n\leq\max_{e\in \tilde{E}_{1}}\tau_e+\min_{x\in \tilde{D}_1(0)}{\rm T}_{ E(\Z^d)\backslash  \tilde{E}_{1}}(x,n\mathbf{e}_1),$$
we obtain
\al{
  & \P\left(\min_{x\in \tilde{D}_1(0)}{\rm T}_{ E(\Z^d)\backslash  \tilde{E}_{1}}(x,n\mathbf{e}_1)>(\mu-\e)n\right)\\
  &\geq \P\left(\forall e\in \tilde{E}_{1},~\tau_e<\frac{\e n}{2},~{\rm T}_n>\left(\mu-\frac{\e}{2}\right)n\right)\\
  &\geq \P\left({\rm T}_n>\left(\mu-\frac{\e}{2}\right)n\right)-\P\left(\exists e\in \tilde{E}_{1} ,~\tau_e\geq \frac{\e n}{2}\right),
}
which converges to $1$ as $n\to\infty$. Therefore, for sufficiently large $n$, we have
\ben{\label{lower-main}
  \P({\rm T}_n>(\mu+\xi)n)\geq \frac{\beta_1^{2d}}{2}\exp{(-2d\alpha (\xi+\e)^r n^r)}.
  }
Since $\e$ is arbitrary, letting $\e\to 0$ after $n\to\infty$, we get
$$\varliminf_{n\to\infty}\frac{1}{n^r}\log{\P({\rm T}_n>(\mu+\xi)n)}\geq -2d\alpha \xi^r.$$

\subsection{Proof of the lower bounds in Theorem \ref{th:rLessThandBig1} and Theorem \ref{th:requalsd}}\label{section: lower bound r>1}
Let us first introduce some notations that we will use repeatedly in the paper. We let
{
\begin{equation} \label{eq:defellIntro}
\ell_M(n)= \ell_{d,r,M}(n)=
\begin{cases}
\left\lceil M n^{\frac{r-1}{d-1}}\right\rceil&\text{ if }1<r<d,\\
~\vspace{-4mm}&\\
\left\lceil \frac{Mn}{1+\log{n}}\right\rceil&\text{ if }r=d,\\
~\vspace{-4mm}&\\
\left\lceil Mn\right\rceil&\text{ if }r>d,\\
\end{cases}
\end{equation}
where $\lceil x \rceil = \inf\{n\geq x,n\in \mathbb N\}$ and $M>0$ is arbitrary}.
We also set $${\rm T}^{[1]}_n = {\rm T}_{D_{\ell_M(n)}(0)}(0,\partial D_{\ell_M(n)}(0))\text{ and }{\rm T}^{[2]}_n = {\rm T}_{D_{\ell_M(n)}(n \mathbf{e}_1)}(n\mathbf{e}_1,\partial D_{\ell_M(n)}(n \mathbf{e}_1)).$$  

In this section, we now restrict ourselves to the case $1<r\leq d$. We consider the following events:
\[F_1 = \left\{{\rm T}_n^{[1]} \geq \frac{(\xi+\varepsilon) n}{2}\right\}, \quad F_2 = \left\{ {\rm T}_n^{[2]} \geq \frac{(\xi+\varepsilon) n}{2}\right\},\]
and
\[G = \left\{\min_{x\in \partial D_{\ell_M(n)}(0),y\in \partial D_{\ell_M(n)}(n \mathbf{e}_1) } {\rm T}_{D_{\ell_M(n)}(0)^c \cap D_{\ell_M(n)}(n \mathbf{e}_1)^c}(x,y) \geq (\mu - \varepsilon) n \right\}.\]
Note that for $n$ large enough, we have $D_{\ell_M(n)}(0)\cap D_{\ell_M(n)}(n \mathbf{e}_1) = \emptyset$. It is then straightforward to check that $F_1 \cap F_2 \cap G \subset \{{\rm T}_n \geq (\mu + \xi) n\}$, and thus by the independence structure,
\begin{equation} \label{eq:lowerBoundFG}
\P\left({\rm T}_n\geq (\mu+\xi) n \right) \geq \P(F_1)^2 \,\P(G).
\end{equation} 

We begin with estimating $\P(F_1)$. This will give the main contribution. 
Take $(t_e^\star)_{e\in E_{\ell_M(n)}}$ that minimizes $\lambda^P_{d,r}(D_{\ell_M(n)}(0),\partial D_{\ell_M(n)}(0))$ (see the definitions around \eqref{def:lambda}). In particular,
\[
\left\{\forall e\in E_{\ell_M(n)},  \tau_e \geq \frac{(\xi + \varepsilon)n}{2} t_e^\star\right\} \subset  \left\{{\rm T}_n^{[1]} \geq \frac{(\xi+\varepsilon) n}{2}\right\} = F_1.\]
On the other hand, using \eqref{cond:distr},

\begin{align*}
  &\P\left(\forall e\in E_{\ell_M(n)},  \tau_e  \geq \frac{(\xi + \varepsilon)n}{2} t_e^\star\right)\\
   &=  \prod_{e\in E_{\ell_M(n)}}\P\left( \tau_e  \geq \frac{(\xi + \varepsilon)n}{2} t_e^\star\right) \\
  &\geq \beta_1^{|E_{\ell_M(n)}|} \exp{\left(-\alpha \left(\frac{(\xi + \varepsilon)n}{2}\right)^r \sum_{e\in E_{\ell_M(n)}} (t_e^\star)^r \right)}\\
& = \beta_1^{|E_{\ell_M(n)}|} \exp{\left(-\alpha \left(\frac{(\xi + \varepsilon)n}{2}\right)^r \lambda_{d,r}(\ell_M(n))\right)},
\end{align*}
where we have used \eqref{def:lambdaIntro} in the last line.

When $r<d$, we have $\frac{d(r-1)}{d-1}<r$ and hence $|E_{\ell_M(n)}|\leq C_d n^{\frac{d(r-1)}{d-1}} =o(n^r)$. This yields
\begin{equation} \label{eq:lowBPF1rlessd}
\varliminf_{n\to\infty} \frac{1}{n^r} \log \P(F_1) \geq -\alpha  \left(\frac{\xi + \varepsilon}{2}\right)^r \lim_{n\to\infty} \lambda_{d,r}(\ell_M(n)). 
\end{equation}
When $r=d$, we have $|E_{\ell_M(n)}| \leq C_d \frac{n^{d}}{(\log n)^{d}}=o(n^{d} \lambda_{d,d}(\ell_M(n)))$ by Theorem~\ref{asymptotic for lambda general}, so that
\begin{equation} \label{eq:lowBPF1}
\begin{aligned}
&\varliminf_{n\to\infty} \frac{1}{n^d \lambda_{d,d}(\ell_M(n))} \log \P(F_1)\geq -\alpha \left(\frac{\xi + \varepsilon}{2}\right)^d. 
\end{aligned}
\end{equation}

We now turn to estimating $\P(G)$. For all $x\in \partial D_{\ell_M(n)}(0)$, consider some path $\gamma_x$ inside $D_{\ell_M(n)}(0)$ that leads from $0$ to $x$ such that $|\gamma_x| \leq 2d \ell_M(n)$,  and let  $$U_n^{[1]} = \sup_{x\in \partial D_{\ell_M(n)}(0)} \sum_{e\in \gamma_x} \tau_e.$$ 
Define $U_n^{[2]}$ similarly with $0$ replaced by $n\mathbf{e}_1$. We have,
\[{\rm T}_n \leq U_n^{[1]} + U_n^{[2]} + \min_{x\in \partial D_{\ell_M(n)}(0),y\in \partial D_{\ell_M(n)}(n \mathbf{e}_1)} {\rm T}_{D_{\ell_M(n)}(0)^c \cap D_{\ell_M(n)}(n\mathbf{e}_1)^c}(x,y),\]
Therefore,
\begin{align}
\P(G) & \geq \P\left ({\rm T}_n  \geq \left(\mu -  \frac{\varepsilon}{2}\right)n,  U_n^{[1]} < \frac{\varepsilon}{4}n, U_n^{[2]} < \frac{\varepsilon}{4}n\right) \nonumber\\
&\geq \P\left({\rm T}_n \geq \left(\mu - \frac{\varepsilon}{2}\right)n \right) - \P\left(U_n^{[1]} \geq \frac{\varepsilon}{4}n \text{ or } U_n^{[2]} \geq \frac{\varepsilon}{4}n\right)\nonumber\\
&\geq \P\left({\rm T}_n \geq \left(\mu - \frac{\varepsilon}{2}\right)n \right) - 2\P\left(U_n^{[1]} \geq \frac{\varepsilon}{4}n\right). \label{eq:lowBdPG}
\end{align}
By the union bound and the exponential Markov inequality,
\begin{align*}
\P\left(U_n^{[1]} \geq \frac{\varepsilon}{4}n\right) & \leq  C_d (\ell_M(n))^{d-1}\sup_{x\in \partial D_{\ell_M(n)}(0)} \P\left(\sum_{e\in \gamma_x} \tau_e \geq \frac{\varepsilon}{4}n\right) \\
&  \leq  C_d (\ell_M(n))^{d-1} \exp{\left(c 2 d \ell_M(n) - \frac{\varepsilon}{4}n\right)} \to 0,
\end{align*}
as $n\to\infty$, with $c=\log \E[e^{\tau_e}]$.
Since the the first term of \eqref{eq:lowBdPG} converges to $1$, it follows that $\lim_{n\to\infty} \P(G) = 1$.

Hence, { when $r\in (1,d)$}, by \eqref{eq:lowerBoundFG} and \eqref{eq:lowBPF1rlessd}, letting $n\to\infty$ followed by $\varepsilon\to 0$,  we obtain:
\begin{equation} \label{lower bound for d<r}
 \varliminf_{n\to\infty} \frac{1}{n^r} \log \P\left({\rm T}(0,nx) \geq (\mu+\xi) n \right)\geq -\alpha  2^{1-r} \xi^r \lambda_{d,r},
 \end{equation}
that is the lower bound in Theorem \ref{th:rLessThandBig1}.
When $r=d$, the estimates  \eqref{eq:lowerBoundFG} and \eqref{eq:lowBPF1} yield that
\begin{equation}
\begin{aligned}
 \varliminf_{n\to\infty} \frac{1}{n^d\,\lambda_{d,d}(\ell_M(n))} \log \P\big({\rm T}(0,nx)\geq (\mu+\xi) n \big)
\geq  -\alpha  2^{1-d} \, \xi^d.\label{lower bound for d=r}
\end{aligned}
\end{equation}
which gives a lower bound for Theorem \ref{th:requalsd} {(note that by Theorem \ref{th:upcoming} we can replace $\lambda_{d,d}(\ell_M(n))$ by $\lambda_{d,d}(n)$ since both are equivalent)}.

\subsection{Proof of Theorem~\ref{asymptotic of the rate function}}
 { We suppose that $r>d$.} We fix $\xi>0$ and $n\in\N$. We consider $\ell_{\delta \xi}(n)=\lceil \delta \xi n\rceil$, where $\delta$ is chosen later.   Let $${\rm T}^{[1]}_n = {\rm T}_{D_{\ell_{\delta \xi}(n)}(0)}(0,\partial D_{\ell_{\delta \xi}(n)}(0)).$$  We consider the following events:
\[F = \left\{{\rm T}_n^{[1]} \geq 2\xi n\right\},\,G = \left\{\min_{x\in \partial D_{\ell_{\delta \xi}(n)}(0) } {\rm T}_{D_{\ell_{\delta \xi}(n)}(0)^c}(x,n\mathbf{e}_1) \geq (\mu - \xi) n \right\}.\]
The same argument as in Section~\ref{section: lower bound r>1} shows
\al{
\P(F)&\geq \P(\forall e\in E_{\ell_{\delta \xi}(n)},\,\tau_e>2\delta^{-1})\\
&\geq  \exp{\left(-c_{\delta} \ell_{\delta \xi}( n)^d \right)}\geq \exp{\left(-c'_{\delta} \xi^{d} n^d\right)},
}
 with some constants $c_\delta,c'_\delta>0$ depending on $\delta$ and $\P(G)\geq 1/2$ if we take $\delta$ small enough. As in Section~\ref{section: lower bound r>1}, we have
\begin{equation} \label{eq:lowerBoundr>d}
\P({\rm T}_n>(\mu+\xi)n)\geq \P(F)\P(G) \geq
\frac{1}{2}\exp{\left(-c'_{\delta} \xi^{d} n^d\right)},
\end{equation}
which implies $\overline{I}(\xi)\leq c'_{\delta} \xi^{d}$. It yields the theorem.
\section{The upper bound}
As before, we write ${\rm T}_n={\rm T}(0,n\mathbf{e}_1)$ and  $\mu=\mu(\mathbf e_1)$. We begin by a lemma that will play a crucial role. The lemma is a variant of \cite[Lemma 3.1]{CZ03} -- in fact, the case $r\geq 1$ is proved therein. We prove it in Appendix \ref{appendixA}. 
\begin{lem}\label{lem:Zhang}
 Assume \eqref{UB condition}. For any $\e>0$, there exist $K=K(\e)\in\N$ and a positive constant $c=c(\e)$ such that for any $n\in \N$,
\[\P\left({\rm T}_{\R\times[-K,K]^{d-1}}\left(0,n\mathbf{e}_1\right)\geq (\mu+\e)n\right)\leq\exp{(-cn^{r\land 1})}.
\]
\end{lem}
\subsection{Proof of the upper bound in Theorem \ref{thm:main1} ($r\leq 1$)}

Let $\e\in(0,\xi)$ and take $K=K(\e)\in\N$ and $c=c(\e)>0$ as in Lemma~\ref{lem:Zhang}. Let $M=M(\xi,\e,c,K)\in 3K\N$ so that
\ben{\label{choice:M} \frac{cM}{K}>12d\alpha\xi^r.}
Given $v\in \mathbb{Z}^{d-1}$ and $n\in\N$, we define a slab  as
\[\S_v=\S_v(K)=\mathbb R \times (v+[-K,K]^{d-1}),\]
and write
\[v^{[1]}_n=(0,v)\text{ and }v^{[2]}_n=(n,v).\]
See also Figure \ref{figure2}. Letting \[{\rm B}_{K,M}=3K\Z^{d-1}\cap [-M,M]^{d-1},\]
we note that $v^{[1]}_n\in D_M(0)$ and $v^{[2]}_n\in D_M(n\mathbf{e}_1)$ whenever $v\in {\rm B}_{K,M}$.

  For $v\neq w\in {\rm B}_{K,M}$, since $\S_v$ and $\S_w$ are disjoint, ${\rm T}_{\S_v}\left(v^{[1]}_n,v^{[2]}_n\right)$ and ${\rm T}_{\S_w}\left(w^{[1]}_n,w^{[2]}_n\right)$ are independent. Moreover, for $v\in {\rm B}_{K,M}$, $\S_v$ is congruent with $\mathbb R\times[-K,K]^{d-1}$. 
  By Lemma~\ref{lem:Zhang} and \eqref{choice:M}, since ${|{\rm B}_{K,M}|}\geq M/3K$, we obtain that for $n\in\N$,
  \al{
\P\left(\forall v\in {\rm B}_{K,M},~{\rm T}_{\S_v}\left(v^{[1]}_n,v^{[2]}_n\right)\geq (\mu+\e)n\right)  &\leq \exp{\left(-cn^r{|{\rm B}_{K,M}|}\right)}\\
  &\leq \exp{(-4d\alpha \xi^r n^r)}.
}
Thus, by the union bound,
\begin{equation}
\begin{aligned}
  & \P({\rm T}_n>(\mu+\xi)n)\\
%
 &\leq \P\left({\rm T}_n>(\mu+\xi)n,~\exists v\in {\rm B}_{K,M}\text{ s.t. }{\rm T}_{\S_v}\left(v^{[1]}_n,v^{[2]}_n\right)<(\mu+\e)n\right)\\
  &\hspace{8mm}+\exp{(-4d\alpha \xi^r n^r)}.
  \end{aligned}
  \label{eq:middleNegl}
 \end{equation}

{The next lemma shows that under the event appearing in the right-hand side of \eqref{eq:middleNegl}, there are two points $x,y$ close to the origin and $n\mathbf{e}_1$ such that the sum of the first passage times from $0$ to $x$ and from $y$ to $n\mathbf{e}_1$ are at least of the order of $\xi n$. This last event will give the main contribution to the large deviation probability $\P({\rm T}_n>(\mu+\xi)n)$.}
\begin{lem}
  Suppose that ${\rm T}_n>(\mu+\xi)n$ and that there exists $v\in {\rm B}_{K,M}$\text{ such that }$${\rm T}_{\S_v}\left(v^{[1]}_n,v^{[2]}_n\right)<(\mu+\e)n.$$ Then, there exist $x\in D_M(0)$ and $y\in D_M(n\mathbf{e}_1)$ such that
\[
{\rm T}(0,x)+{\rm T}(y,n\mathbf{e}_1)\geq (\xi-\e)n.
\]
\end{lem}
  \begin{proof}
    Let $v\in {\rm B}_{K,M}$ be such that ${\rm T}_{\S_v}\left(v^{[1]}_n,v^{[2]}_n\right)<(\mu+\e)n$. 
     By the triangular inequality,
    \al{
      (\mu+\xi)n< {\rm T}_n
      &\leq {\rm T}(0,v^{[1]}_n)+{\rm T}(v^{[2]}_n,n\mathbf{e}_1)+{\rm T}_{\S_v}\left(v^{[1]}_n,v^{[2]}_n\right)\\
      &< {\rm T}(0,v^{[1]}_n)+{\rm T}(v^{[2]}_n,n\mathbf{e}_1)+(\mu+\e)n.
      }
   Thus,
    $${\rm T}(0,v^{[1]}_n)+{\rm T}(v^{[2]}_n,n\mathbf{e}_1)\geq (\xi-\e)n,$$ and $x=v^{[1]}_n$ and $y=v^{[2]}_n$ are the desired objects.
  \end{proof}
  Using the lemma above,
  \aln{
    & \P\left({\rm T}_n>(\mu+\xi)n,~\exists v\in {\rm B}_{K,M}\text{ s.t. }{\rm T}_{\S_v}\left(v^{[1]}_n,v^{[2]}_n\right)<(\mu+\e)n\right)\nonumber\\
    &\leq \P(\exists{}x\in D_M(0),~\exists y\in D_M(n\mathbf{e}_1)\text{ s.t. }{\rm T}(0,x)+{\rm T}(y,n\mathbf{e}_1)\geq (\xi-\e)n)\nonumber\\
    &\leq \sum_{x\in D_M(0)}\sum_{y\in D_M(n\mathbf{e}_1)}\P({\rm T}(0,x)+{\rm T}(y,n\mathbf{e}_1)\geq (\xi-\e)n).\label{comp3}
}
  To estimate the inside of the summation, for $n>8M$, we consider $4d$ disjoint paths $\{r^x_i\}^{2d}_{i=1}\subset D_{2M}(0)$ from $0$ to $x$ and $\{r^y_i\}^{2d}_{i=1}\subset D_{2M}(n\mathbf{e}_1)$ from $y$ to $n\mathbf{e}_1$ so that
  $$\max\{{|r^{z}_i|}:~i\in\{1,\cdots,2d\},~z\in\{x,y\}\}\leq 4dM,$$
  where ${|r|}$ is the number of edges in a path $r$, as in \cite[p 135]{Kes86}. We get
\al{
  &\qquad \P({\rm T}(0,x)+{\rm T}(y,n\mathbf{e}_1)\geq (\xi-\e)n) \\
  &\leq \P(\forall i\in\{1,\cdots,2d\},~{\rm T}(r_i^x)+{\rm T}(r^y_i)\geq (\xi-\e)n)\nonumber\\
  &= \prod_{i=1}^{2d}\P\left(\sum_{e\in r_i^x\cup r_i^y}\tau_e \geq (\xi-\e)n\right)\nonumber\\
  &\leq \exp{(-2d(1-\e)\alpha ((\xi-\e) n)^r )},
}
where we have used Lemma~\ref{exp:est} below in the last line.
\begin{lem}\label{exp:est}
  Let $(X_i)_{i=1}^k$ be identical{ly} and independent{ly} { distributed} satisfying \eqref{UB condition}  with $r\in (0,1]$. Then for any $c\in (0,1)$ there exists $n_0=n_0(k,c)$ such that for any $n\geq n_0$, 
 \[\P\left(\sum_{i=1}^k X_i>n\right)\leq  \exp{(-(1-c)\alpha n^r)}.\]
\end{lem}
\begin{proof}
  Since $\E \exp{(a X_1^r)}<\infty$ for $a<\alpha$ and
  \ben{\label{concave}
    \left(\displaystyle\sum_{i=1}^k x_i\right)^r\leq \displaystyle\sum_{i=1}^k x_i^r\text{ for $x_i\geq 0$},
    }we obtain by the exponential Markov inequality, for {any $c\in (0,1)$}, for sufficiently large $n$,
  \al{
    \P\left(\sum_{i=1}^k X_i>n\right)&\leq \exp{(-(1-c/2)\alpha n^r)}(\E \exp{((1-c/2) \alpha X_1^r)})^k\\
    &\leq \exp{(-(1-c)\alpha n^r)}.
    }
\end{proof}
Therefore, \eqref{comp3} can be bounded from above by
\al{
  & \sum_{x\in D_M(0)}\sum_{y\in D_M(n\mathbf{e}_1)}\exp{(-2d(1-\e)\alpha((\xi-\e) n)^r)}\\
  &\leq (4M)^{2d}\exp{(-2d(1-\e)\alpha((\xi-\e) n)^r)}{ .}\\
}
Since $\e$ is arbitrary, letting $\e\to 0$ after $n\to\infty$, we get by \eqref{eq:middleNegl} that
$$\varlimsup_{n\to\infty}\frac{1}{n^r}\log{\P({\rm T}_n>(\mu+\xi)n)}\leq -2d\alpha \xi^r.$$

\subsection{Proof of Theorem \ref{prop1}}

We fix $r > 1$ throughout this section. We prove Theorem \ref{prop1} by induction on $d\geq 1$. In order to use the induction hypothesis, we will rely on the two following technical lemmas.
\begin{lem}\label{prop LDP} Assume that the conclusion of Theorem \ref{prop1} holds for some $d\geq 1$.
   Then there exist $c_2=c_2(d,r),c=c(d,r)>0$ such that for all $L_1\in \mathbb N$, $L_2\geq L_1$ and $v,w\in[0,L_1]^d$,
  \begin{equation} \label{eq:lemCorInduction}
\P\left( {\rm T}_{[0,L_1]^d}(v,w)> c_2 L_2 \right)\le \exp{(-c\,{\rm g}_{d,r}(L_1,L_2))}.
  \end{equation}
\end{lem}

\begin{proof}[Proof of Lemma \ref{prop LDP}]
Let $v(0)=v$ and $$v(i)=(w_1,\cdots,w_i,v_{i+1},v_{i+2},\cdots,v_d).$$ We have,
  \begin{equation} \label{eq:decv}
  \begin{split}
    \P\left( {\rm T}_{[0,L_1]^d}(v,w)> c_2L_2\right)\le \sum^d_{i=1}\P \left({\rm T}_{[0,L_1]^d}(v(i-1),v(i))> d^{-1} c_2L_2 \right).
  \end{split}
  \end{equation}
  First consider the term for $i=1$ and let
  $$l=|v_1-w_1|/2=|v(1)-v(0)|_{{\infty}}{/2}\leq L_1/2.$$
  Suppose that $v_1<w_1$ and $v+[0,2l]\times [0,l]^{d-1} \subset [0,L_1]^d$. We have,
\[
\P ({\rm T}_{[0,L_1]^d}(v(0),v(1))> d^{-1} c_2L_2 )\le \P ({\rm T}_{[0,2l]\times [0,l]^{d-1}}(0,2l \mathbf{e}_1)> d^{-1} c_2L_2 ),\]
which is less than or equal to
\begin{align*}
  & \P \left({\rm T}_{[0,2l]\times [0, l ]^{d-1}}(0, \lfloor l \rfloor \mathbf{e}_1)> (2d)^{-1} {c_2L_2} \right)\\
  &\qquad\qquad+\P \left({\rm T}_{[0,2l]\times [0,l]^{d-1}}(\lfloor l \rfloor \mathbf{e}_1, 2l \mathbf{e}_1)> (2d)^{-1}{c_2L_2}\right )\\
  &\leq   \P \left({\rm T}_{[0,\lfloor l \rfloor]^d}(0, \lfloor l \rfloor\mathbf{e}_1)> (2d)^{-1}{c_2L_2}\right )\\
  &\qquad\qquad+ \P \left({\rm T}_{[0,2l-\lfloor l \rfloor]^d}(0, (2l-\lfloor l \rfloor)\mathbf{e}_1)> (2d)^{-1}{c_2L_2}\right ).
\end{align*}
If we now let $c_2$ be large enough  so that $N:=\sqrt{c_2} L_2 \geq l+1$ and $(2d)^{-1}\sqrt {c_2} \geq \mu(\mathbf e_1)+1$, then the estimate \eqref{eq:thInduction} yields that for $n=\lfloor l\rfloor$ or $n= 2l - \lfloor l \rfloor$,
\begin{align*}
\P \left({\rm T}_{[0,n]^d}(0,n\mathbf{e}_1)> (2d)^{-1} c_2L_2 \right)& \leq \exp{\left(-c_{r,d} {\rm g}_{d,r}(n,N)\right)}\\
&\leq \exp{\left(- c_{r,d}'{\rm g}_{d,r}(L_1,L_2)\right)},
\end{align*}
for some $c'_{r,d}>0$, where the last inequality holds taking $c_2$ large enough, by using in particular that $n\to g(n,N)$ is non-increasing and that $n\leq L_1$.

In general, there exists a rectangle $z+[0,2l]\times[0,l]^{d-1}$ for $z\in\mathbb Z^d$, {included in $[0,L_1]^d$,} such that $v$ and $v(1)$ are corners of the rectangle, and the previous argument applies up to rotations.
    The same can be done for $v(2),\dots,v(d)$, so using \eqref{eq:decv} this entails \eqref{eq:lemCorInduction}.
\end{proof}

\begin{lem} \label{rem:face} Assume that the conclusion of Lemma \ref{prop LDP} holds for dimension $d-1$ with $d\geq 1$.
Consider $J$ to be a union of faces of the cube $[0,L]^{d}$ and assume that $J$ is a connected set. There exist $c_2=c_2(d,r)>0$ and $c=c(d,r)>0$, such that for all $L'\geq L$ {with $L\in \mathbb N$}, 
\begin{equation} \label{eq:LDPface}
\P\left(\text{$\exists v,w\in J$  s.t. ${\rm T}_{J}(v,w)> c_2 L'$}\right) \leq  \exp{\left(-
c\,{\rm g}_{d-1,r}\left(L,L'\right)\right)}.
\end{equation}
\end{lem}
\begin{proof}
{Let $c_2$ be as in Lemma \ref{prop LDP}}.
We have
\begin{equation*}
\P\left(\text{$\exists v,w\in J$  s.t. ${\rm T}_{J}(v,w)> 2d c_2 L'$}\right) \leq |J|^2 \sup_{v,w\in J}\P\left({\rm T}_{J}(v,w)>2d c_2 L'\right).
\end{equation*}

To bound the supremum above, fix $v,w\in J$. Considering a shortest path from $v$ to $w$ with respect to the graph distance in $J$, one can see that there exist $l\leq 2d$ and $v(0)=v, v(1),\cdots,v(l)=w \in J$ which belong to the shortest path, such that for any $ i\in \Iintv{0, l}$, $v(i)$ and $v(i-1)$ are on the same face of $J$ and for any $ i\in \Iintv{1, l-1}$, $v(i)$ lies on an edge of $J$. Let ${\rm W}_i$ be the face which contains $v(i)$ and $v(i-1)$.  Since {$W_i$} is congruent with $[0,L]^{d-1}$, the estimate \eqref{eq:lemCorInduction}
yields
\begin{equation*}
  \begin{split}
\P\left( {\rm T}_{J}(v,w)> 2d c_2 L'\right) & \leq  \P\left(\exists i \in \Iintv{1,l} \text{ s.t. } {\rm T}_{{\rm W}_i}(v(i-1),v(i))> c_2 L' \right)\\
    &\leq  2d\exp{(-c\,{\rm g}_{d-1,r}(L,L'))},
      \end{split}
\end{equation*}
for some $c=c(d,r)>0$, which shows Lemma \ref{rem:face}.
\end{proof}

The following lemma plays an important role in the proof of Theorem \ref{prop1}. A proof of the lemma is presented \cite[Theorem 1.1]{LW09} in a more general framework, but it can alternatively be proved directly using Chernoff's bound.

\begin{lem}\label{LDP}
   Let $\{X_i\}_{i\in\N}$ be independent and such that $X_i$ has same distribution as $\tau_e$. If $r>1$, then there exist $c,s>0$ such that for $n\in \mathbb N$ and $N\geq s n$, 
\begin{equation} \label{eq:standardLDE}
\P\left(\sum^{n}_{i=1} X_i\geq N \right)\leq{}\exp{\left(-c n^{1-r} N^r \right)}.
\end{equation}
\end{lem}

From the previous lemma, we can obtain the following refinement of Lemma \ref{lem:Zhang}. Recall $\mu= \mu(\mathbf e_1)$.
\begin{lem} \label{lem:refinedLemma} Suppose $d\geq 2$. If $r>1$, then for any $\e>0$, there exist $K=K(\e)\in\N$ and a positive constant $c=c(\e)$ such that for any $n\in \mathbb N$ and $N \geq n$,
\begin{equation} \label{eq:refinedLemma}
\P\left({\rm T}_{\mathbb R \times[-K,K]^{d-1}}\left(0,n\mathbf{e}_1\right)\geq (\mu+\e)N\right)\leq \exp{\left(-c n^{1-r} N^r \right)}.
\end{equation}
\end{lem}

\begin{proof}
Let $\{X_i\}$ and $s>0$ be from Lemma \ref{LDP}. First suppose that $(\mu+\e)N \geq s n$. The LHS of \eqref{eq:refinedLemma} is bounded from above by
\[
\P\left(\sum^{n}_{i=1} X_i\geq (\mu+\e)N \right)\leq{}\exp{\left(-c (\mu+\varepsilon)^r  n^{1-r} N^r \right)}.
\]
On the other hand, when $(\mu + \e)N < s n$, since $N\geq n$, the LHS of \eqref{eq:refinedLemma} is bounded from above by 
\al{
 \P\left({\rm T}_{\R\times[-K,K]^{d-1}}\left(0,n\mathbf{e}_1\right)\geq (\mu+\e)n\right)&\leq \exp{\left(-c n\right)}\\
   &\leq \exp{\left(-c' n^{1-r} N^r \right)},
  }
  for some $c'=c'(\mu,s,\e)>0$, where we have
taken $K=K(\e)$ and $c=c(\e)$ from Lemma \ref{lem:Zhang} in the first inequality.  Putting things together, the two estimates give \eqref{eq:refinedLemma}.
\end{proof}

\begin{proof}[Proof of Theorem \ref{prop1}]
  We proceed by induction on dimension $d$, for a fixed $r>1$.

  We start with $d=1$.  In this case, it is enough to show that for $(X_i)$ independent and distributed as $\tau_e$, for all $\e>0$, there exists $c=c(\e)$ such that for all $N\geq n$,
\begin{equation} \label{eq:dequals1}
\P\left(\sum_{i=1}^n X_i \geq (\mu + \e)N\right) \leq \exp{\left(-cn^{1-r} N^r\right)},
\end{equation}
where $\mu = \E[\tau_e]$ in this case. We can use similar arguments to the proof of Lemma \ref{lem:refinedLemma}. Let $s$ be from Lemma \ref{LDP}. If $(\mu+\e)N \geq s n$, then
\[
\P\left(\sum^{n}_{i=1} X_i\geq (\mu+\e)N \right)\leq{}\exp{\left(-c (\mu+\varepsilon)^r  n^{1-r} N^r \right)}.
\]
On the other hand, when $(\mu + \e)N < s n$ and since $N\geq n$, the LHS of \eqref{eq:dequals1} is bounded from above by 
\al{
  \P\left(\sum_{i=1}^n X_i \geq (\mu+\e)n\right)&\leq \exp{\left(-c n\right)}\\
  &=  \exp{\left(-c n \frac{N^r} {n^r} \frac{n^r}{N^r}\right)} \leq \exp{\left(-c' n^{1-r} N^r \right)},
  }
for some $c=c(\e)$ coming from standard large deviation estimates for i.i.d. random variables, and some $c'=c'(\e,s,\mu)>0$. The last two estimates entail \eqref{eq:dequals1}.

We now suppose Theorem \ref{prop1} holds for $d-1\geq 1$ and we show that it also holds for $d$.
{  Recall $\ell_{\delta}(n)$ for $\delta>0$ from \eqref{eq:defellIntro}.} We always assume that $\delta$ is small enough so that $2 \ell_{\delta}(n)\leq n$.
 Given $v\in\Z^{d-1}$, we write
\[v^{[1]}_{n}=(0,v)\text{ and }v^{[2]}_{n}=(n,v).\]
Let ${\rm B}_{K,\delta,n}^+=3K\Z^{d-1}\cap [1, \ell_{\delta}(n)]^{d-1}$ and $\mathbb S_v = \mathbb{R}\times(v+[-K,K]^{d-1})$ for $v\in {\rm B}_{K,\delta,n}^+$. 
We first consider the event:
\[B_1=\left\{\exists v \in {\rm B}_{K,\delta,n}^+,{\rm T}_{\mathbb S_v}\left(v^{[1]}_{n},v^{[2]}_{n}\right) \leq (\mu + \varepsilon)N\right\}.\]
For $\varepsilon > 0$, let $K=K(\varepsilon)$ as in Lemma~\ref{lem:refinedLemma}. We have $|{\rm B}_{K,\delta,n}^+| \geq  \left(\frac{\ell_{\delta}(n)}{6K}\right)^{d-1}$, so by the independence structure and translation invariance, it follows { from Lemma~\ref{lem:refinedLemma}} that
\begin{equation} \label{eq:tunnels_L1}
\begin{aligned}
  \P({}^cB_1)& \leq \exp{\left(-c' K^{-d+1} \left(\ell_{\delta}(n) \right)^{d-1} n^{1-r} N^r\right)}\\
 & \leq \exp{\left(- c' K^{-d+1} \delta^{d-1} {\rm g}_{d,r}(n,N)\right)},
\end{aligned}
\end{equation}
with some $c'>0$. Now for $1\leq i \leq 2\ell_\delta(n)$, let
\[
J^{[1]}_i=\partial D_{i}(0)\cap [0,n]^{d},
\]
which is the union of the faces of the cube $[0,i]^d$ that do not contain $0$.
Let also 
$$J^{[2]}_i =\partial D_{i}(n\mathbf{e}_1)\cap  [0,n]^{d},$$
which is the translation of $J^{[1]}_i$ for which $n \mathbf e_1$ is a corner of the cube. (See also Figure \ref{figure}). 
\begin{figure}[h]
    \centering
   \includegraphics[width=1\textwidth]{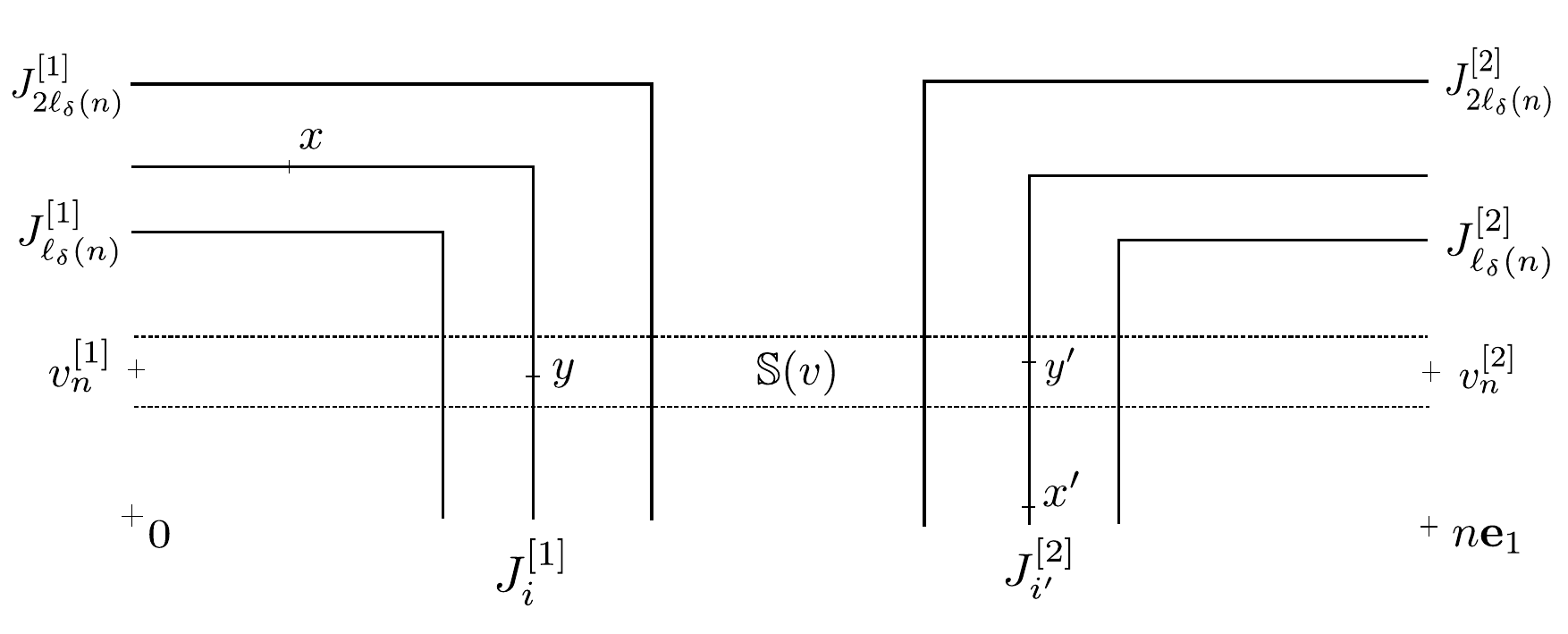}
    \caption{}
    \label{figure}
\end{figure} We then consider the events:
\[A_1^{[1]}= \left\{\exists i \in \Iintv{\ell_\delta(n)+1, 2\ell_\delta(n)} : \forall x,y\in J^{[1]}_{i}, {\rm T}_{J^{[1]}_{i}}(x,y) \leq \e N\right\},
\]
and $A_1^{[2]}$ defined similarly with $J^{[1]}_i$ replaced by $J^{[2]}_i$.

%
We now estimate $\P\left({}^cA_1^{[1]}\right)$. 
%
By the induction hypothesis, Lemma \ref{prop LDP} and Lemma \ref{rem:face} with $L=i$ and $J=J_i^{[1]}$, there exists $c_2=c_2(d,r),c=c(d,r)>0$ such that \eqref{eq:LDPface} holds for $L'\geq i$. We choose $L'=2 \lceil \e N c_2^{-1}\rceil$ and let $\delta<\e c_2^{-1}$ so that $L'\geq i$ (use that $i\leq 2 \ell_\delta(n) \leq 2 \lceil \delta n\rceil\leq 2 \lceil\delta N\rceil $), in order to obtain that
\begin{align*}
\P\left(\exists x,y\in J^{[1]}_{i}, {\rm T}_{J^{[1]}_{i}}(x,y) > \e N \right) & \leq  \exp{\left(- c\,{\rm g}_{d-1.r}(i,L')\right)}\\
& \leq  \exp{\left(- c_{\e,r,d}\,{\rm g}_{d-1.r}(2\ell_{\delta}(n),N)\right)},
\end{align*}
for some constant $c_{\e,r,d}>0$ {and for $N\geq N_0(\varepsilon)$ which ensures that $L'\leq 4\varepsilon N c_2^{-1}$}. 
Since the $J^{[1]}_i$'s do not intersect, we further obtain  by independence that
\begin{equation}  \label{eq:polyFactorL1}
\P\left({}^cA_1^{[1]}\right) \leq  \exp\left\{ - c_{\e,r,d} \,\ell_{\delta}(n) {\rm g}_{d-1.r}(2\ell_{\delta}(n), N)\right\}, 
\end{equation}
where
\begin{align}
\ell_{\delta}(n) {\rm g}_{d-1.r}(2\ell_{\delta}(n),N) \geq \delta  \times \begin{cases}
 n^{\frac{r-1}{d-1}}  N^r & \text{if } 1<r<d-1,\\
n^{\frac{r-1}{d-1}} \frac{N^r}{(1+\log n )^{d-2}} & \text{if } r=d-1,\\
 4^{d-1-r}n^{\frac{r-1}{d-1}}  n^{\frac{(d-1-r)(r-1)}{d-1}} N^r & \text{if } d-1<r<d,\\
 \frac{ n^{d-r}}{1+\log n}  N^r  & \text{if } r=d,\\
n^{d-r}  N^r& \text{if } r>d,
 \end{cases}
\end{align}
where we have used that $2\ell_\delta(n)\leq n$ for $\delta$ small enough in the second, fourth {and fifth} cases, {and that $2\ell_\delta(n) \leq 4n^{\frac{r-1}{d-1}}$ in the third case}. 
We obtain in any case that there is some $c_{d,\delta,\e}>0$, such that for $N$ large enough and any $n\leq N$,
\begin{equation} \label{bound_PAcL1}
\P\left({}^cA^{[1]}_1\right) \leq \exp{\left(- c_{\e,d,\delta} \,  {\rm g}_{d,r}(n,N)\right)}.
\end{equation}

 { For any $u>0$,} let $D^+_u(z)=D_u(z)\cap [0,n]^d$ and $D^+_u=D^+_u(0)$. We observe the following: suppose that $A^{[1]}_1\cap A_1^{[2]} \cap B_1$ holds and let $i,i',v$ satisfy the conditions in $A^{[1]}_1,A_1^{[2]},B_1$. Consider $x\in J^{[1]}_i$ that belongs to a path that minimizes 
\[{\rm T}^{[1]}_{n}:={\rm T}_{D^+_{2\ell_\delta(n)}}\left(0, J^{[1]}_{2\ell_\delta(n)}\right),\] and $x' \in J^{[2]}_{i'}$ that belongs to a path minimizing
\[{\rm T}^{[2]}_{n}:={\rm T}_{D_{2\ell_\delta(n)}^+(n\mathbf e_1)} \left(n\mathbf{e}_1,J^{[2]}_{2\ell_\delta(n)} \right).\]
 
 Let $\gamma$ be a path that minimizes $T_{\mathbb S_v}\big(v^{[1]}_{n},v^{[2]}_{n}\big)$. Since $v\in {\rm B}_{K,\delta,n}^+\subset [0,\ell_\delta(n)]^{d-1}$, there exists $y\in J^{[1]}_{i} \cap \mathbb S_v$ and $y'\in J^{[2]}_{i'} \cap \mathbb S_v$ such that there is a sub-path of $\gamma$ that we call $\tilde{\gamma} = (\gamma_{i_0},\dots,\gamma_{j_0})$, that goes from $\gamma_{i_0}=y$ to $\gamma_{j_0}=y'$, and which is included in $[0,n]^d$ (see Figure \ref{figure}). In particular, such a path satisfies
 \[
{\mathrm T_{[0,n]^d}(y,y') \leq \sum_{e\in \tilde{\gamma}} \tau_e \leq \mathrm T_{\mathbb S_v}\big(v^{[1]}_{n},v^{[2]}_{n}\big) \leq (\mu + \varepsilon) N.
 }
 \]
By the triangular inequality, we obtain that
\begin{align*}
&{\rm T}_{[0,n]^d}(0, n \mathbf{e}_1)\\
 & \leq  {\rm T}_{[0,n]^d}(0,x) + {\rm T}_{[0,n]^d}(x,y) + {\rm T}_{[0,n]^d}(y, y') + {\rm T}_{[0,n]^d}(y',x') + {\rm T}_{[0,n]^d} (x',n\mathbf{e}_1)\\
&\leq {\rm T}^{[1]}_{n} + \e N + (\mu + \e) N + \e N + {\rm T}^{[2]}_{n}.
\end{align*}
Therefore, on $A^{[1]}_1\cap A_1^{[2]}\cap B_1$, the event $\{{\rm T}_{[0,n]^d}(0,n \mathbf{e}_1) > (\mu+\xi) N\}$ implies that
\[
 {\rm T}^{[1]}_{n}  +  {\rm T}^{[2]}_{n}> (\xi - 3\e)N,
\]
and thus, 
\begin{equation} \label{eq:sumIndependentL1}
\begin{aligned}
& \P({\rm T}_{[0,n]^d}(0,n \mathbf{e}_1)  >(\mu + \xi) N) \\
&\leq \P\left({}^cA^{[1]}_1\right) + \P\left({}^cA^{[2]}_1\right) + \P({}^cB_1) 
+ \P\left( {\rm T}^{[1]}_n + {\rm T}^{[2]}_n  \geq (\xi-3\e) N\right)\\
&\leq 2\P\left({}^cA^{[1]}_1\right) + \P({}^c B_1) 
+2\P\left( {\rm T}^{[1]}_{n}  \geq \frac{\xi-3\e}{2} N\right).
\end{aligned}
\end{equation}

We now focus on estimating $\P\left( {\rm T}^{[1]}_{n}\geq u \right)$ for all $u>0$. 
 For ${ k}\in \mathbb N$, let 
   $E^+_{k}$ be the set of all edges that have both ends in $D^+_k=D_k(0)\cap [0,n]^d$. 
We define
\[H_u^+=\left\{(t_e)_{e\in E^+_{2\ell_\delta(n)}} \in (\mathbb R_+)^{E^+_{2\ell_\delta(n)}}:~\forall \gamma:0\to  J^{[1]}_{2\ell_\delta(n)}\text{ with }\gamma\subset D^+_{2\ell_\delta (n)},\,\sum_{e\in \gamma} t_e \geq u\right\}.
\]
 Let $Y$ be a random variable with density $\alpha e^{-\alpha t}\mathbf{1}_{t\geq 0}$ and $X=(Y+\rho)^{1       /r}$. Then $X$ has density $e^{\alpha \rho} \alpha rt^{r-1}e^{-\alpha t^r}$ and for $\rho$ large enough,
\[\P(\tau_e > t) \leq \beta e^{-\alpha t^r} \leq e^{\alpha \rho} e^{-\alpha t^r} =  \P(X \geq t),\]
so that $\tau_e$ is stochastically dominated by $X$. Since $H_u^+$ is an increasing event with respect to the $t_e$'s, we obtain that
\begin{align}
&\P\left(
{\rm T}^{[1]}_{n} \geq u\right)\nonumber\\
 &\leq \int_{H_u^+} \prod_{e\in  E^+_{2\ell_\delta(n)}} e^{\alpha \rho} \alpha r t_e^{r-1} e^{-\alpha t_e^r} \mathrm d t_e \nonumber\\
&\leq \left( \int \prod_{e\in  E^+_{2\ell_\delta(n)}} e^{\alpha \rho} \alpha r t_e^{r-1} e^{-\alpha \varepsilon t_e^r} \mathrm d t_e \right) e^{-\alpha(1-\varepsilon)\inf_{(t_e)\in H_u^+}\left\{\sum_{e \in E^+_{2\ell_\delta(n)}} t_e^r \right\}}  .\label{twofactors}
 \end{align}
The first factor on the last line writes
\al{ 
  \left(\int_{\mathbb R_+} e^{\alpha \rho} \alpha r t_e^{r-1} e^{-\alpha\varepsilon t^r} {\rm d} t\right)^{|E^+_{2\ell_\delta(n)}|}& \leq \exp\left(c'_{\varepsilon,d} |E^+_{2\ell_\delta(n)}|\right)\\
  &\leq \exp\left(c_{\varepsilon,d} \ell_\delta(n)^d\right),
  }
where $c_{\varepsilon,d},c'_{\varepsilon,d}>0$. By relying on Theorem~\ref{thm: path formalization}, the second factor in \eqref{twofactors} writes $e^{-\alpha(1-\varepsilon) u^r \lambda^+_{d,r}(2\ell_\delta(n))}$, where $\lambda^+_{d,r}(k)=\lambda_{d,r}(D_k^+;J_k^{[1]})$ (recall the notation from \eqref{general lambda}). We thus find that
\begin{equation} \label{eq:boundLaplace}
\begin{aligned}
&\P\left( {\rm T}^{[1]}_{n}    \geq \frac{(\xi-3\e)}{2}N\right)\\
 &  \leq  \exp{\left(c_{\varepsilon,d} \ell_{\delta}(n)^d-\alpha(1-\varepsilon) \left(\frac{(\xi-3\e)N} 2\right)^r  \lambda^+_{d,r}(2\ell_{\delta}(n))\right)}.
 \end{aligned}
\end{equation}

{In order to obtain a lower bound on $\lambda^+_{d,r}(k)$, we use the following inequality involving $\lambda_{d,r}(k)$. We can then rely on Theorem~\ref{asymptotic for lambda general} which gives the order of magnitude of the latter quantity.}
\begin{lem} \label{lem: tilda lambda and lambda}
  For any $d\geq 2$ and $r>0$,
  \begin{equation}\label{tilda lambda and lambda}
    \lambda^+_{d,r}(k) \geq 2^{-d} \lambda_{d,r}(k).
    \end{equation}
\end{lem}
\begin{proof}
  Given a function $f^+:D_k^+\to \R$ with $f^+(0)=0$ and $f(x)\geq 1$ for $x\in J_k^{[1]}$, we define $f:D_k(0)\to \R$ such that
  $$f(x_1,\cdots,x_d)=f^+(|x_1|,\cdots,|x_d|).$$
  Then for any $x\in \partial D_k(0)$, $f(x)\geq 1$ and $f(0)=0$. Moreover,
  \al{
    \sum_{\langle x,y\rangle \in E_k^+} |f(x)-f(y)|^r\geq 2^{-d}    \sum_{\langle x,y\rangle \in E_k} |f(x)-f(y)|^r.
  }
  Hence, we have \eqref{tilda lambda and lambda}.
  \end{proof}
From \eqref{eq:boundLaplace}, Theorem~\ref{asymptotic for lambda general} and Lemma~\ref{lem: tilda lambda and lambda}, we find that there exists $c_{\e,d}{,c_{d,r}}>0$ such that for $N$ large enough,
\al{
  &\qquad \P\left( {\rm T}^{[1]}_{n} \geq{ \frac{(\xi-3\e)}{2}}N\right)\\
  &\leq 
\begin{cases}
 \exp{\left( c_{\varepsilon,d} \delta^d n^{d \frac{r-1}{d-1}}-\alpha(1 - \varepsilon) ((\xi-3\e)/2)^r N^r c_{d,r}\right)} & \text{ if } r<d,\\
  \exp{\left( c_{\varepsilon,d} \frac{\delta^dn^{d}}{(1+\log n )^{d}}-\alpha(1 - \varepsilon) ((\xi-3\e)/2)^r N^d  \frac{c_{d,r}}{(1+\log \ell_\delta(n)  )^{d-1}}\right)} & \text{ if } r=d,\\
 \exp{\left(c_{\varepsilon,d} \delta^d n^{d}-\alpha(1 - \varepsilon) ((\xi-3\e)/2)^r N^r c_{d,r} {\lceil\delta n \rceil ^{d-r}}\right)} & \text{ if } r>d.
\end{cases}
}
Using that $n\leq N$, that  $d\frac{r-1}{d-1}< r$ in the case $r<d$ and that $\log \lceil \delta n\rceil \leq \log n $ in the case $r=d$, we finally obtain that for $\delta$ and $\e$ small enough (which are now fixed), 
\begin{equation}  \label{eq:Laplace_L1}
\P\left( {\rm T}^{[1]}_{n}  \geq { \frac{(\xi-3\e)}{2}} N\right) \leq \exp{\left(-c_{d,\xi} {\rm g}_{d,r}(n,N)\right)},
\end{equation}
for some $c_{d,\xi}>0$.

Thanks to \eqref{eq:sumIndependentL1}, \eqref{bound_PAcL1}, \eqref{eq:tunnels_L1} and \eqref{eq:Laplace_L1}, we can conclude that there exists $c=c(\xi,r,d)>0$ such that for $N$ large enough and all $n\leq N$,
\[\P\left({\rm T}_{[0,n]^d}(0,{n}\mathbf{e}_1)> (\mu(\mathbf{e}_1)+\xi) N\right) \le \exp{(-c\,{\rm g}_{d,r}(n,N))}.
\]
To finally obtain \eqref{eq:thInduction} (which holds for any $N\geq n$), we observe that for all $N_0>0$,
\[\max_{ N\leq N_0}\max_{n\leq N}\P\left({\rm T}_{[0,n]^d}(0,{n}\mathbf{e}_1)> (\mu(\mathbf{e}_1)+\xi) N\right)<1,\]
so up to reducing the constant $c$, \eqref{eq:thInduction} holds for all $N\geq n$.
\end{proof}

\subsection{Proof of the upper bounds in Theorem \ref{th:rLessThandBig1} and Theorem \ref{th:requalsd} (case $1<r\leq d$)} \label{subsec:proofUpperBound}
The proof is very similar to the one of Theorem \ref{prop1}, the main difference being that we will use Theorem \ref{prop1} as an input to obtain estimates on large deviations in dimension $d-1$. In particular, we do not need to follow an induction argument here.
Recall $\ell_M(n)=\ell_{d,r,M}(n)$ from \eqref{eq:defellIntro}. In the following, we will take $M$ big, but we always assume that $n$ is large enough (depending maybe on $M$) so that $2\ell_{M}(n)\leq n$. Let 
\[{\rm B}_{K,M,n}=3K\Z^{d-1}\cap [-\ell_{M}(n), \ell_{M}(n)]^{d-1}.\]
 Given $v\in\Z^{d-1}$, we write
\[v^{[1]}_n=(0,v)\text{ and }v^{[2]}_n=(n,v).\]
For all $v\in {\rm B}_{K,M,n}$, we have $v^{[1]}_n\in D_{\ell_{M}(n)}(0)$ and $v^{[2]}_n\in D_{\ell_{M}(n)}(n \mathbf{e}_1)$. First consider the event:
\[B_2=\left\{\exists v \in {\rm B}_{K,M,n},{\rm T}_{\mathbb S_v}(v^{[1]}_n,v^{[2]}_n) \leq (\mu + \varepsilon)n\right\},\]
where $\mathbb S_v = \mathbb{R}\times(v+[-K,K]^{d-1})$. 
For $\varepsilon > 0$, let $K=K(\varepsilon)$ as in Lemma~\ref{lem:Zhang} and let $M=M(K,\xi,a,\alpha)>K$ that will be determined later. We have $|{\rm B}_{K,M,n}| \geq  \left(\frac{\ell_{M}(n)}{3K}\right)^{d-1}$, so by the independence structure and translation invariance, it follows { from Lemma~\ref{lem:Zhang}} that  with $c=c(\varepsilon)$,
\begin{equation} \label{eq:tunnels}
\P({}^cB_2) \leq \exp{\left(-cn \left(\frac{\ell_{M}(n)}{3K}\right)^{d-1}\right)}.
\end{equation}
 Let us introduce two additional events of interest. Let 
\[I_i^{[1]} = 
\partial D_i(0),\  I_i^{[2]} = \partial D_i(n\mathbf{e}_1),\quad  0\leq i \leq 2\ell_{M}(n), \] 
 which are boundaries of cubes of length $ i$. (See also Figure \ref{figure2}). 
\begin{figure}[h]
    \centering
    \includegraphics[width=1\textwidth]{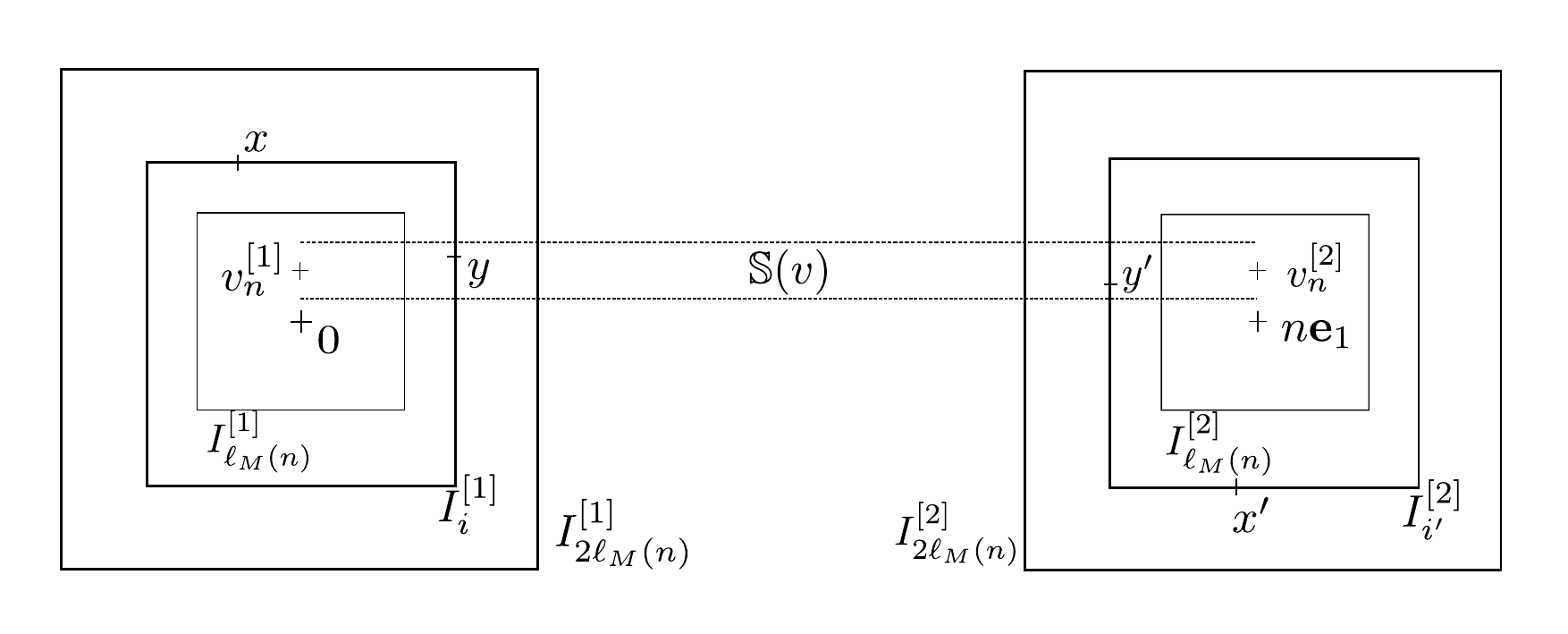}
    \caption{}
    \label{figure2}
\end{figure}  We consider the events:
\[A_2^{[1]}= \left\{\exists i \in \Iintv{ \ell_{M}(n)+1,2 \ell_{M}(n)} : \forall x,y\in I_{i}^{[1]}, {\rm T}_{I_{i}^{[1]}}(x,y) \leq \varepsilon n\right\},
\]
and $A^{[2]}_2$ defined similarly with $I_i^{[1]}$ replaced by $I_i^{[2]}$.

We now estimate $\P({}^cA_2^{[1]})=\P({}^cA_2^{[2]})$.  By Lemma~\ref{rem:face} {(whose assumption is verified thanks to Theorem \ref{prop1})}, there exists a positive constants $C_{\varepsilon,d,r}$ such that
\begin{equation*}
\P\left(\exists x,y\in I_{i}^{[1]}, {\rm T}_{I_{i}^{[1]}}(x,y) > \varepsilon n\right) \leq \exp{\left(-{\rm g}_{d-1,r}(2\ell_M(n),n) C_{\varepsilon,d,r}\right)}.
\end{equation*}
By independence, we further obtain that
\begin{align} 
\P({}^cA_2^{[1]}) & \leq  \exp{\left(-{\rm g}_{d-1,r}(2\ell_M(n),n) C_{\varepsilon,d,r}\ell_{M}(n)\right)} \nonumber\\
&\leq C_{d,M} \times \begin{cases}
 \exp{\left( -C_{\varepsilon,d,r}' M n^{r}\right)} & \text{if } r<d,\\
 \exp{\left( -C_{\varepsilon,d,r}' M \frac{n^{d}}{(\log n)^{d-1}}\right)} & \text{if } r=d,
 \end{cases}
  \label{eq:boundPAc}
\end{align}
for $n$ large enough and some  constants $C_{\varepsilon,d,r}',C_{d,M}>0$, and where we have used in the last inequality that when $r\in [d-1,d)$ and $d\geq 2$, we have $\frac{r-1}{d-1} + d-1 \geq r$ by standard computations.

  We now make the following observation: suppose that $A_2^{[1]}\cap A_2^{[2]} \cap B_2$ holds and let $i,i',v$ satisfy the conditions in $A_2^{[1]},A_2^{[2]},B_2$. Consider $x\in I_i^{[1]}$ (resp.\ $x' \in I_{i'}^{[2]}$) that belongs to a path that minimizes 
\[{\rm T}^{[1]}_n:={\rm T}_{D_{2\ell_M(n)}(0)}\left(0,I_{2\ell_M(n)}^{[1]}\right)\] respectively\ 
\[{\rm T}^{[2]}_n:={\rm T}_{D_{2\ell_M(n)}(n\mathbf{e}_1)}\left(n\mathbf{e}_1,I_{2\ell_M(n)}^{[2]}\right).\] Let also $y\in I_i^{[1]} \cap \mathbb S_v$ and $y'\in I_{i'}^{[2]} \cap \mathbb S_v$. By the triangular inequality,
\begin{align*}
{\rm T}(0,n \mathbf{e}_1) & \leq  {\rm T}(0,x) + {\rm T}(x,y) + {\rm T}(y, y') + {\rm T}(y',x') + {\rm T} (x',n\mathbf{e}_1)\\
&\leq {\rm T}^{[1]}_n + \varepsilon n + (\mu+\varepsilon)n + \varepsilon n + {\rm T}^{[2]}_n.
\end{align*}
Therefore, on $A_2^{[1]}\cap A_2^{[2]}\cap B_2$, the event $\{{\rm T}(0,n \mathbf{e}_1) > (\mu+\xi)n\}$ implies that
\begin{equation}\label{eq:underA1A2B2}
 {\rm T}^{[1]}_n  +  {\rm T}^{[2]}_n> (\xi-3\varepsilon)n.
\end{equation}
Thus,
\begin{equation} \label{eq:sumIndependent}
\begin{aligned}
& \P({\rm T}_n  > (\mu + \xi)n) \\
&\leq \P({}^cA_2^{[1]}) + \P({}^cA_2^{[2]}) + \P({}^cB_2) 
+ \P\left( {\rm T}^{[1]}_n + {\rm T}^{[2]}_n  \geq (\xi - 3\varepsilon) n\right)\\
&\leq 2\P({}^cA_2^{[1]}) + \P({}^cB_2) 
+\sum_{k=0}^{\lf(\xi - 3\varepsilon) n\rf} \P\left( {\rm T}^{[1]}_n\geq k \right) \P({\rm T}^{[2]}_n  \geq (\xi - 3\varepsilon) n-k-1),
\end{aligned}
\end{equation}
 where in the last inequality we have used independence of ${\rm T}^{[1]}_n$ and ${\rm T}^{[2]}_n$, the union bound and that $X+Y\geq x$ implies $X\geq \lfloor X\rfloor$ and $Y\geq x- \lfloor X\rfloor-1$.

Similarly to \eqref{eq:boundLaplace}, we obtain that
\begin{equation*} 
\qquad \P\left( {\rm T}^{[1]}_{n}    \geq k \right)
   \leq  \exp{\left(c_{\varepsilon,d} \ell_{M}(n)^d\right)} \exp{\left(-\alpha(1-\varepsilon) k^r  {\lambda}_{d,r}(2\ell_{M}(n))\right)},
\end{equation*}
where $\lambda_{d,r}$ is defined in \eqref{def:lambda}. The same estimate holds for ${{\rm T}}^{[2]}_n$.
Therefore the last sum in \eqref{eq:sumIndependent} is bounded from above by
\begin{align}
&e^{2c_{\varepsilon,d}  \ell_{M}(n)^d} \sum_{k=0}^{\lf(\xi - 3\varepsilon) n\rf} e^{-\alpha(1-\varepsilon) k^r    {\lambda}_{d,r}(2\ell_{M}(n))}e^{-\alpha(1-\varepsilon) ((\xi - 3\varepsilon)n-k)^r \lambda_{d,r}(2\ell_M(n))} \nonumber\\
& \leq \xi n \, \exp{\left(2c_{\varepsilon,d}  \ell_{M}(n)^d-\alpha(1-\varepsilon) 2\left(\frac{\xi - 3\varepsilon}{2}\right)^r n^r \lambda_{d,r}(2\ell_M(n))\right)}, \label{eq:Laplace}
\end{align}
where we have used that $(s\xi)^r + ((1-s)\xi)^r\geq 2 (\xi/2)^r$. 

We observe that  by
 \eqref{eq:tunnels} and for $M$ large enough, $\P({}^cB_2)$ is negligible with respect to $\P({\rm T}_n  > (\mu + \xi)n)$ because of \eqref{lower bound for d<r} and \eqref{lower bound for d=r}. Then, when $r<d$, we find by \eqref{eq:boundPAc}
that $\P({}^c A_2^{[1]})$ can be made smaller than $e^{-Cn^r}$ for all $C>0$ by increasing $M$. Finally, since $\ell_M(n)^d = o(n^r)$, we can neglect the first exponential term in \eqref{eq:Laplace} and obtain the upper bound in Theorem \ref{th:rLessThandBig1} from \eqref{eq:sumIndependent}, \eqref{eq:Laplace} and Theorem \ref{asymptotic for lambda general}. When $r=d$, $\P({}^c A_2^{[1]})=o(e^{-C n^d/(\log n)^{d-1}})$ for {any given} $C>0$ {if we choose $M$ large enough}. The upper bound in Theorem \ref{th:requalsd} follows from \eqref{eq:Laplace} {(note that by Theorem \ref{th:upcoming} we can replace $\lambda_{d,d}(2\ell_M(n))$ by $\lambda_{d,d}(n)$ since both are equivalent)}, where we can neglect the first exponential term since $n^{-d} (\log n)^{d-1} \ell_M(n)^d \to 0$ and by Theorem \ref{asymptotic for lambda general}.

\begin{remark}  \label{rem: general x}
We have seen the proofs for $x=\mathbf{e}_1$.  We remark here how to prove the main results for general $x$. For the lower bound, the proof is exactly the same. For the upper bound, the proof is the same as before except for Lemma~\ref{lem:Zhang}. Then we replace Lemma~\ref{lem:Zhang} by the following:
\begin{lem}
  {
    Given $x\in \R^d\backslash \{0\}$, let
    $$\S^x(K)=\{y\in\R^d:~\text{$\exists t\in\R$ s.t. }|y-tx|_{\infty}\leq K\}.$$
     For any $\e>0$, there exist $K\in\N$ and $c>0$ such that for any $n\in \N$,
\[\P\left({\rm T}_{\S^x(K)}\left(0,nx\right)\geq (\mu+\e)n\right)\leq\exp{(-cn^{r\land 1})}.
\]}
  \end{lem}
  The proof of the lemma is essentially the same as in Lemma~\ref{lem:Zhang} and we safely leave it for readers. 
  \end{remark}

\section{Proof of Theorem~\ref{asymptotic for lambda general}}\label{section: sketch}
 We write $D_M=D_M(0)$ and $\partial D_M=\partial D_M(0)$ for simplicity of notation. 
\subsection{Upper bound}
We first prove that $\kappa_{d,r}$ is bounded. 

For $r<d$, we consider $f(0)=0$ and $f(x)=1$ for $x\neq 0$. Then, $f(x)\geq 1$ for $x\in\partial D_n$ and 
$$\sum_{\langle x,y\rangle\in E_n}|f(x)-f(y)|^r=2d.$$ 
Hence $\kappa_{d,r}$ is less than or equal to $2d$, in particular bounded.

Next we consider $r=d$. Let us set $f(x)=\frac{\log{(|x|_1+1)}}{\log{n}}$. Then for $x\in\partial D_n$, $f(x)\geq 1$ and $f(0)=0$. Moreover, for $x\in D_n$, 
\al{
|f(x+\mathbf{e}_i)-f(x)|&\leq \frac{1}{\log{n}}(\log{(|x|_1+2)}-\log{|x|_1})\\
&\leq \frac{C}{ (|x|_1+1) \log{n}},
}
with some constant $C>0$ independent of $x$ and $n$. Hence, using $r=d$, we have
\al{
\sum_{x,y\in E_n} |f(x)-f(y)|^r&\leq 2d C^r (\log{n})^{-r} \sum_{x\in D_n}  (|x|_1+1)^{-r}\\
&\leq C' (\log{n})^{-r} \sum_{\ell=1}^{2dn}  \ell^{d-1}\ell^{-r}\\
&\leq C'' (\log{n})^{-(d-1)},
}
  with some constants $C',C''>0$, which proves the upper bound.

Finally, we consider $r>d$. Let us set $f(x)=n^{-1} |x|_1$ for $x\neq 0$ and $f(0)=0$. Then for $x\in\partial D_n$, $f(x)\geq 1$. Moreover, for $x\in D_n$,
\al{
|f(x+\mathbf{e}_i)-f(x)|\leq  n^{-1},
}  
with some $C>0$. Hence,
\al{
\sum_{x,y\in E_n} |f(x)-f(y)|^r&\leq \frac{2d}{n^{r}} |D_n|\\
&\leq C' n^{d-r},
}
  with some constant $C'>0$, which proves the upper bound.

\subsection{Lower bound} \label{subsec:lowerBoundPath}
We prove that $\kappa_{d,r}$ is bounded away from $0$. We use $\lambda_{d,r}(n)=\lambda^{\rm P}_{d,r}(n):=\lambda^{\rm P}_{d,r}(D_n;\partial D_n)$ from Theorem~\ref{thm: path formalization}.  Given an edge $e=\langle x,y\rangle\in E(\Z^d)$, we write $|e|_1=|x|_1\lor |y|_1$.

We first suppose that $r>1$. Given $x\in \partial D_n$, we consider the line $L_x$ connecting $0$ and $x$. Then we take a path $\gamma_x$ from $0$ to $x$ to be the nearest neighbor closest to $L_x$ with a deterministic rule braking ties. The choice of $\gamma_x$ is not so important, but the following lemma is crucial, and is straightforward to check.
  \begin{lem}\label{lem: path probab}
    Let us consider a uniform random variable $X$ on $\partial D_n$ and let $\mu_n$ be its measure. Then, there exists $C>0$ independent of $n$ such that  
    $${p_e:=}\mu_n(e\in \gamma_X) \leq
    \begin{cases}
           C |e|^{-(d-1)}_1&\text{ for any $e\in E_n$ },\\
           0& \text{ for any $e\notin E_n$ }.
           \end{cases}$$
  \end{lem}
  
  Suppose $(t_e)_{e\in E_n}$ satisfies that for any $\gamma:0\to \partial D_n$, $\sum_{e\in \gamma} t_e\geq 1$. 
  Then, 
  \al{
    1\leq \mu_n\left[ \sum_{e\in \gamma_X} t_e  \right]= \sum_{e\in E(\Z^d)}p_e t_e.
  }
  Hence, for any $\lambda\geq 0$,
  \al{
    \sum_{e\in E_n}t_e^r&\geq  \sum_{e\in E_n}t_e^r -\lambda \left(\sum_{e\in E_n}p_e t_e-1\right)\\
    &\geq \inf_{(t_e)\in \R^{E_n}_+}\left\{ \sum_{e\in E_n}t_e^r -\lambda \left(\sum_{e\in E_n}p_e t_e-1\right) \right\}.
    }
  This means that
  $$\lambda^{\rm P}_{d,r}(n)\geq \sup_{\lambda\geq 0} \inf_{(t_e)\in \R^{E_n}_+}\left\{ \sum_{e\in E_n}t_e^r -\lambda \left(\sum_{e\in E_n}p_e t_e-1\right) \right\}. $$
  Fixing $\lambda\geq 0$, we  have
  \aln{
    \inf_{(t_e)\in \R^{E_n}_+}\left\{ \sum_{e\in E_n}t_e^r -\lambda \left(\sum_{e\in E_n}p_e t_e-1\right) \right\}&\geq   \left(\sum_{e\in E_n} \inf_{t_e\in \R_+}\left\{t_e^r -\lambda p_e t_e\right\}\right) +\lambda\notag\\
    &= -C_n \frac{r-1}{r^{\frac{r}{r-1}}} \lambda^{\frac{r}{r-1}}+\lambda, \label{Saddle}
  }
  where $C_n= \sum_{e\in E_n} p_e^{\frac{r}{r-1}}$ and we have used the saddle point method in the last line. Putting $\lambda=r\,C_n^{1-r}$ into \eqref{Saddle}, we conclude
\begin{equation} \label{eq:LowerBoundLambdaP}
    \lambda_{d,r}(n)=\lambda^{\rm P}_{d,r}(n)\geq C_n^{1-r}.
\end{equation}
  By Lemma~\ref{lem: path probab}, it is straightforward to check that
  \al{
    C_n\leq
    \begin{cases}
      C_{d,r}&\text{ if $1<r<d$},\\
      C_{d,r} \log{n} &\text{ if $r=d$},\\
      C_{d,r} n^{\frac{r-d}{r-1}} &\text{ if $r>d$},
      \end{cases}
  }
  with some constant $C_{d,r}>0.$ This proves the theorem for $r>1$.

 { For $r\leq 1$, since $r\rightarrow \lambda_{d,r}(n)$ is non-increasing, we get the result.}

  \section{Proof of Theorem \ref{th:rLessThandBig1Loc}}
   We restrict { ourselves} to the case $x=\mathbf e_1$, although the same idea applies to any $x\in { \R^d\backslash\{0\}}$.
We begin with \eqref{eq:Noloc}.  { Suppose $r\geq d$.} Let $a>0$, $\varepsilon_0>0$. There exists $c=c(a,\varepsilon_0)>0$ such that for $n$ large enough,
  \[\P\left(\exists e\in E_{n^a}, \tau_e> \varepsilon_0 n  \right) \leq e^{-c n^r}.\]
When $r>d$, by \eqref{eq:lowerBoundr>d}, there exists $C=C_{\xi}>0$ such that for $n$ large enough, $$\P({\rm T}(0,n\mathbf e_1) > (\mu + \xi)n)\geq e^{-C n^{d}},$$ and when $r=d$, by Theorem \ref{th:requalsd} and Theorem \ref{asymptotic for lambda general}, $$\P({\rm T}(0,n\mathbf e_1) > (\mu + \xi)n)\geq e^{-C \frac{n^{d}}{(\log n)^{d-1}}}.$$   Estimate \eqref{eq:Noloc} follows directly in both cases.
  
We turn to \eqref{eq:loc}. { Suppose $r< d$}. We repeat the main steps of the proof exposed in Section \ref{subsec:proofUpperBound} and we use its notations. {We observe that \eqref{eq:loc} follows if we can show that $Q_n^i \to 0,i=1,2$ where
\begin{equation*} \label{eq:needToShow}
Q_n^i =  \frac{\P\left(\forall e\in E_R, \tau_e\leq \varepsilon_0 n\, , \, {\rm T}^{[i]}_{n}\geq \frac{(\xi-3\e)n}{2} \right)}{\P\left(\mathrm T_n  \geq (\mu+\xi) n\right)}.
\end{equation*}
Indeed, by \eqref{eq:underA1A2B2}, the left-hand side of \eqref{eq:loc} is smaller than
\[
Q_n^1+Q_n^2 + \frac{\P\Big({}^cA_2^{[1]}\Big) + \P\Big({}^cA_2^{[2]}\Big) + \P({}^cB_2) }{\P\left(\mathrm T_n  \geq (\mu+\xi) n\right)},
\]
where the second term vanishes as $n\to\infty$, for $M$ chosen large enough, by the arguments given in the paragraph above Remark \ref{rem: general x}. We will now show that $Q_n^1\to 0$. (note that $Q_n^2 \to 0$ can be treated similarly).
}
 We write $D_M=D_M(0)$ and denote by $E_{R+1,n}$ the set of edges in the annulus 
$D_n\setminus D_R$, i.e.\ $E_{R+1,n} = E_n \setminus E_R$.
Now suppose that for all $ e\in E_R$ we have $\tau_e\leq \varepsilon_0 n$, then for all $x\in \partial D_R$, 
\begin{align*}
{\rm T}^{[1]}_{n} & \leq {\rm T}_{D_R}(0,x) + {\rm T}_{E_{R+1,{2\ell_M(n)}}}(x,\partial D_{{2\ell_M(n)}})\\
& \leq 2dR{ \varepsilon_0} n + {\rm T}_{E_{R+1,{2\ell_M(n)}}}(x,\partial D_{{2\ell_M(n)}}),
\end{align*}
so if we define ${\rm T}_{{A}}(F,G) = \inf_{x\in F,y\,\in G} {\rm T}_{ A}(x,y)$, then we obtain
\[{\rm T}^{[1]}_{n} \leq 2dR\varepsilon_0 n + {\rm T}_{E_{R+1,{2\ell_M(n)}}}(\partial D_R,\partial D_{{2\ell_M(n)}}).
\]
Therefore,
\begin{align}
& \P\left(\forall e\in E_R, \tau_e\leq \varepsilon_0 n\, , \, {\rm T}^{[1]}_{n}\geq \frac{(\xi-3\e)n}{2} \right) \nonumber\\
& \leq \P\left( {\rm T}_{E_{R+1,{2\ell_M(n)}}}(\partial D_R,\partial D_n) \geq \frac{(\xi-3\e-4dR\varepsilon_0)n}{2} \right) \nonumber\\
& \leq e^{c_{\varepsilon,d} \ell_{M}(n)^d} e^{-(\alpha-\varepsilon) \left(\frac{(\xi-3\e-4dR\varepsilon_0)n}{2}\right)^r  {\lambda}_{d,r}^P(E_{R+1,{2\ell_M(n)}};\partial D_R,\partial D_{{2\ell_M(n)}})}, \label{eq:upperBoundDn}
\end{align}
by a similar computation to \eqref{eq:boundLaplace}, where: 
\begin{align*}
&\lambda_{d,r}^{\rm P}(E_{R+1,n};\partial D_R,\partial D_n)\\
& = \inf_{(t_e)\in (\mathbb R_+)^{E_{R+1,n}}} \left\{\sum_{e\in E_{R+1,n}} t_e^{r}\ \middle|~\forall \gamma:\partial D_R\to \partial D_n \text{ with $\gamma\subset E_{R+1,n}$},\,\sum_{e\in\gamma}t_e\geq 1\right\}.
\end{align*}
 Following the argument { around \eqref{eq:LowerBoundLambdaP}} in Section \ref{subsec:lowerBoundPath}, we find that
\begin{equation*}
 { \lambda_{d,r}^{\rm P}(E_{R+1,n};\partial D_R,\partial D_n)} \geq C_{R+1,n}^{1-r}, \quad C_{R+1,n}= \sum_{e\in E_{R+1,n}} p_e^{\frac{r}{r-1}},
\end{equation*}
where $p_e$ is defined in Lemma \ref{lem: path probab}. When $r<d$, we have $C_{R+1,n}\leq c_{d,r} R^{\frac{d-r}{1-r}}$ so that
\[\lim_{R\to\infty} \inf_{{n >  R}}  { \lambda_{d,r}^{\rm P}(E_{R+1,n};\partial D_R,\partial D_n)} \to \infty.\]
Recall $\lambda_{d,r}$ from {\eqref{eq:def_lambdaLimit} which is finite by Theorem \ref{asymptotic for lambda general}}.  We can thus fix $R$ large enough so that
\[{\liminf_{n\to\infty} \lambda_{d,r}^{\rm P}(E_{R+1,n};\partial D_R,\partial D_{2\ell_M(n)})} > {2 \lambda_{d,r}},\] 
 and hence \eqref{eq:upperBoundDn} entails 
{that $Q_n^1\to 0$}
 by Theorem \ref{th:rLessThandBig1}, by choosing $\varepsilon$ and $\varepsilon_0$ small enough.

\appendix
\section{Proof of Lemma~\ref{lem:Zhang}} \label{appendixA}
In this section, we show Lemma~\ref{lem:Zhang}: for any $\e>0$, there exist $M=M(\e)\in\N$ and a positive constant $c=c(\e)$ such that for any $n\in\N$,
  $$\P\left({\rm T}_{\R\times[-M,M]^{d-1}}\left(0,n\mathbf{e}_1\right)\geq (\mu+\e)n\right)\leq e^{-cn^{r\land 1}}.$$
\begin{proof}
{ Since $\P(\tau_e\geq t)\leq (\beta e^{\alpha}) e^{- \alpha t^{r\land 1}}$ for $t\geq 0$, without loss of generality, we can suppose $r\leq 1$.} We take a slightly different approach from that in \cite[Lemma 3.1]{CZ03}.
By \eqref{kingman} and ${\rm T}_{[-M,M]^d}\left(0,K\mathbf{e}_1\right)\to {\rm T}_{\mathbb Z^d}\left(0,K\mathbf{e}_1\right)$ a.s.\ as $M\to\infty$, for any $s,\e>0$, there exist $K<M\in\N$ such that
\ben{\label{LLN:box}
\P\left({\rm T}_{[-M,M]^d}\left(0,K\mathbf{e}_1\right)\geq \mu\left(1+\frac{\e}{2}\right)K\right)< s.
}
Given $\e>0$, let
$$s=s(\e)=\frac{\e^2}{16\E\tau_e^2},$$
and $K=K(\e,s)$, $M=M(\e,s)$ in \eqref{LLN:box}. We write
$$\S=\R \times[-M,M]^{d-1}.$$
Let $n>KM$. If we take $\ell=\lf n/(KM)\rf$, then by the triangular inequality,
\al{
  & \P\left({\rm T}_{\S}\left(0,n\mathbf{e}_1\right)\geq (\mu+\e)n\right)\\
    &\leq \P\left({\rm T}_{\S}\left(0,(KM \ell) \mathbf{e}_1\right)+{\rm T}_{\S}\left((KM \ell) \mathbf{e}_1,n\mathbf{e}_1\right)\geq \left(\mu+\frac{\e}{2}\right)KM \ell +\frac{\e}{2}n\right)\\
  &\leq \P\left({\rm T}_{\S}\left(0,(KM \ell) \mathbf{e}_1\right)\geq \left(\mu+\frac{\e}{2}\right) KM \ell\right)+\P\left({\rm T}_{\S}\left((KM\ell) \mathbf{e}_1,n\mathbf{e}_1\right)\geq \frac{\e}{2}n\right),
  }
and, by Lemma~\ref{exp:est} and { $r\leq 1$ that we assumed at the beginning of the proof}, the second term can be bounded from above by $e^{-c n^{r\land 1}}$ with some $c=c(\e,K,M)>0$.
Hence, without loss of generality, we can assume $n$ is divisible by $KM$, say $n=KM\ell$ with $\ell\in\N$. Given $i\in\{0,\cdots,M\ell-1\}$, we define
$${\rm T}^{[i]}={\rm T}_{(iK\mathbf{e}_1+[-M,M]^d)}(iK\mathbf{e}_1,(i+1)K\mathbf{e}_1).$$
It follows from the definition and the triangular inequality that if $|i-j|\geq M$, then ${\rm T}^{[i]}$ and ${\rm T}^{[j]}$ are independent and
$${\rm T}_{\S}\left(0,n\mathbf{e}_1\right)\leq \sum_{i=0}^{M\ell-1} {\rm T}^{[i]}.$$
Thus, by the union bound and $n=KM\ell$,
\al{
  \P\left({\rm T}_{\S}\left(0,n\mathbf{e}_1\right)\geq (\mu+\e)n\right) &\leq \P\left( \sum_{i=0}^{M\ell-1} {\rm T}^{[i]}\geq (\mu+\e)n\right)\\
  &\leq  \P\left( \sum_{m=0}^{M-1}\sum_{i=0}^{\ell-1} {\rm T}^{[Mi+m]}\geq (\mu+\e)KM\ell \right)\\
  &\leq  \sum_{m=0}^{M-1}\P\left( \sum_{i=0}^{\ell-1} {\rm T}^{[Mi+m]}\geq (\mu+\e)K\ell \right).
}
We only consider the case $m=0$, since the other cases can be treated in the same way. Then,
\aln{
  & \P\left( \sum_{i=0}^{\ell-1} {\rm T}^{[Mi]}\geq (\mu+\e)K\ell\right)\nonumber\\
  & = \P\left( \sum_{i=0}^{\ell-1} {\rm T}^{[Mi]}\mathbf{1}_{\{{\rm T}^{[Mi]}< (\mu+(\e/2))K\}}+\sum_{i=0}^{\ell-1} {\rm T}^{[Mi]}\mathbf{1}_{\{{\rm T}^{[Mi]}\geq (\mu+(\e/2))K\}}\geq (\mu+\e)K\ell\right)\nonumber\\
  &\leq \P\left( \sum_{i=0}^{\ell-1}  (\mu+(\e/2))K+\sum_{i=0}^{\ell-1} {\rm T}^{[Mi]}\mathbf{1}_{\{{\rm T}^{[Mi]}\geq (\mu+(\e/2))K\}}\geq (\mu+\e)K\ell\right)\nonumber\\
  &= \P\left( \sum_{i=0}^{\ell-1} {\rm T}^{[Mi]}\mathbf{1}_{\{{\rm T}^{[Mi]}\geq (\mu+(\e/2))K\}}\geq \e K\ell/2\right).\label{comp}
}
For $k\in\N$, let us denote
$${\rm e}{[k]}=\la k\mathbf{e}_1,(k+1)\mathbf{e}_1\ra.$$
Then, since
$${\rm T}^{[Mi]}\leq \displaystyle\sum_{k=KMi}^{KMi+K-1}\tau_{{\rm e}{[k]}},$$ \eqref{comp} can be bounded from above by
\al{
  & \P\left(\sum_{i=0}^{\ell-1}\sum_{k=KMi}^{KMi+K-1}\tau_{{\rm e}{[k]}}\mathbf{1}_{\{{\rm T}^{[Mi]}\geq (\mu+(\e/2))K\}} \geq \e K\ell/2\right)\\
  &=\P\left(\sum_{k=0}^{K-1}\sum_{i=0}^{\ell-1}\tau_{{\rm e}{[KMi+k]}}\mathbf{1}_{\{{\rm T}^{[Mi]}\geq (\mu+(\e/2))K\}} \geq \e K\ell/2\right)\\
  &\leq \sum_{k=0}^{K-1} \P\left(\sum_{i=0}^{\ell-1}\tau_{{\rm e}{[KMi+k]}}\mathbf{1}_{\{{\rm T}^{[Mi]}\geq (\mu+(\e/2))K\}} \geq \e\ell /2\right).
  }
We write $X^{[k]}_i=\tau_{{\rm e}{[KMi+k]}}\mathbf{1}_{\{{\rm T}^{[Mi]}\geq (\mu+(\e/2))K\}}$. Then $(X^{[k]}_i)_i$ are identically and independent non-negative random variables. By \eqref{LLN:box} and the Cauchy--Schwarz inequality,
\al{\E X^{[k]}_i&\leq (\E\tau_{e}^2)^{1/2}\P\left({\rm T}^{[Mi]}\geq (\mu+(\e/2))K\right)^{1/2}\\
  &=(\E\tau_{e}^2)^{1/2}\P\left({\rm T}_{[-M,M]^d}(0,K\mathbf{e}_1)\geq (\mu+(\e/2))K\right)^{1/2}\\
  &< (\E\tau_{e}^2)^{1/2} s^{1/2}= \e/4,
}
and for $t>0$, $\P(X^{[k]}_i>t)\leq \P(\tau_e>t)\leq \beta e^{-\alpha t^r},$ where $\alpha,\beta$ are in \eqref{UB condition}. For $r<1$, by \cite[(4.2)]{GRR}, there exists $c=c(\e,K,M)>0$ such that
\ben{\label{LDP:r<1}
  \P\left(\sum_{i=0}^{\ell-1}X^{[k]}_i \geq \e\ell /2\right) \leq e^{-cn^{r\land 1}}.
}
On the other hand, for $r\geq 1$, a standard large deviation (Cram\'er Theorem) proves \eqref{LDP:r<1}, e.g. \cite[Theorem~2.2.2,~Lemma~2.2.5]{DZ10}. Putting things together, we have that there exists $n_0$ such that for $n\geq n_0$,
$$\P\left({\rm T}_{\R\times[-M,M]^{d-1}}\left(0,n\mathbf{e}_1\right)\geq (\mu+\e)n\right)\leq 2MK e^{-cn^{r\land 1}}\leq e^{-c'n^{r\land 1}}.$$
 By $\mu\leq \E\tau_e$, since the support of the distribution of $\tau_e$ { intersects} $[0,\mu+\epsilon)$, 
 $$\max_{n\in\Iintv{1, n_0}}\P(\rmT_{\R\times[-M,M]^{d-1}}\left(0,n\mathbf{e}_1\right)\geq (\mu+\e)n)<1.$$
Hence we can find $c'$ such that for $n\in \Iintv{1, n_0}$,
 $$\P\left({\rm T}_{\R\times[-M,M]^{d-1}}\left(0,n\mathbf{e}_1\right)\geq (\mu+\e)n\right)\leq e^{-c'n^{r\land 1}},$$
 which completes the proof.
\end{proof}

\section{Proof of (\ref{thm:main2})}\label{slow-vary}
The proof is essentially the same as in Theorem~\ref{thm:main1}. We only touch with the difference.
We use the Karamata representation for a slowly varying function:
\begin{lem}{\cite[Theorem 1.3.1]{BGT87}}\label{lem: Karamata}
 For any slowly varying function  $\ell(t)$, there exist bounded measurable functions $c(t),\e(t)$ and constants $a>0$ and $c\in\R$, such that for any  $t>a$,
\ben{\label{Karamata repre}
\ell(t)=\exp{\left(c(t)+\int_{a}^t \frac{\e(s)}{s} {\rm d}s\right)},
}
 $\lim_{t\to\infty} c(t)=c$ and $\lim_{t\to\infty} \e(t)=0$.
\end{lem}
We use the following lemma, called the local uniformity of slowly varying functions, which directly follows from the Karamata representation:
\begin{lem}{\cite[Theorem~1.2.1]{BGT87}}\label{local unif SV}
 For any slowly varying function  $\ell(t)$, the convergence 
$$\lim_{t\to \infty}\frac{\ell(at)}{\ell(t)}=1,$$
is uniform if $a$ is restricted to a compact interval in $(0,\infty)$.
\end{lem}
 By  Lemma~\ref{local unif SV}, for sufficiently large $n$, the first term of \eqref{lower-est} is bounded from below by
 \al{
   &\qquad \beta_1((\xi+\e)n)^{2d}\exp{(-2d \alpha_1((\xi+\e)n)((\xi+\e)n)^r)}\\
   &\geq \exp{(-2d (1+\e) \alpha_1(n)((\xi+\e)n)^r)}.
   }
Thus, \eqref{lower-main} is replaced by
$$  \P({\rm T}_n>(\mu+\xi)n)\geq \frac{1}{2}\exp{(-2d (1+\e) \alpha_1(n)((\xi+\e)n)^r))}.$$
For the lower bound, the rest is the same as before. For the upper bound, the proof is exactly the same as in Theorem~\ref{thm:main1} except for Lemma~\ref{lem:Zhang} and Lemma~\ref{exp:est}. Lemma~\ref{exp:est} is replaced by the following lemma.

\begin{lem}\label{exp:est2}
  Let $(X_i)_{i=1}^k$ be identical{ly} and independent{ly} { distributed} satisfying \eqref{cond-distr2} with $r\leq1$. Then for any $c>0$ there exists $n_0=n_0(k,c)$ such that for any $n\geq n_0$,
  $$\P\left(\sum_{i=1}^k X_i>n\right)\leq  \exp{(-(1-c)\alpha_2(n) n^r)}.$$
\end{lem}
\begin{proof}

Given $c,\rho>0$, we choose a random variable $Y$ that has a continuous density $C_{c,\rho} e^{-(1-\frac c 4) \alpha_2(t) t^r}\mathbf{1}_{\{t\geq \rho\}}$ with some $C_{c,\rho}>0$. Then, for $\rho>0$ large enough, $C_{c,\rho}\geq 1$ and  $X$ is stochastically { dominated by} $Y$ since $Y\geq \rho$ a.s. and for $t\geq \rho$,
\al{
\P(Y\geq t)&= \int_{t}^\infty C_{c,\rho}\, e^{-(1-\frac c 4) \alpha_2(s) s^r}{\rm d}s\\
&\geq e^{-(1-\frac c 4)\sup_{s\in [t,t+1]} \alpha_2(s) s^r}\\
&\geq \beta_2(t)\, e^{-\alpha_2(t) t^r}\geq \P(X\geq t),
}
where we have used Lemma~\ref{local unif SV} and $\beta_2(t)\ll e^{\frac{c}{4} \alpha_2(t) t^r}$ for $t$ large enough.

 Hence, it suffices to check that for i.i.d. random variables $Y_i$ that have the same law as $Y$, 
$$\P\left(\sum_{i=1}^k Y_i>n\right)\leq  \exp{(-(1-c)\alpha_2(n) n^r)}.$$
By  Lemma~\ref{lem: Karamata}, for any $\e>0$, there exists $t_0>0$ such that for any $t>t_0$,
\ben{\label{SV prop}
\inf_{s\geq t} \alpha_2(s) s^r \geq (1-\e)\alpha_2(t) t^r.
} Hence, for $n$ large enough,
\al{
\P\left(\exists i\in \Iintv{1,k}\text{ s.t. }Y_i> n \right)&\leq k\P(Y> n)\\
&\leq k C_{c,\rho} \int_{n}^\infty e^{-(1-\frac c 4) \alpha_2(t) t^r}{\rm d} t\\
&\leq k C_{c,\rho}\, e^{-(1-\frac c 2)\inf_{t\geq n} \alpha_2(t) t^r} \int_{0}^\infty e^{-\frac c 4 \alpha_2(t) t^r} {\rm d} t\\
&\leq \exp{\left(-(1-\frac{2c}{3}) \alpha_2(n) n^r\right)}.
}
On the other hand, we have
\al{
&\qquad \P\left(\sum_{i=1}^k Y_i>n,\,Y_i\leq n,\,\forall i\in \Iintv{1,k} \right)\\
&\leq \int_{\R^k_+} \prod_{i=1}^k C_{c,\rho} e^{-(1-\frac{c}{4})\alpha(y_i) y_i^r} {\mathbf{1}_{\sum_{i=1}^k y_i\geq n,\,y_i\leq n}} {\rm d} \mbox{\boldmath $y$}\\
&\leq (C_{c,\rho})^k\, e^{-\left(1-\frac{c}{2}\right) \inf\{\sum_{i=1}^k \alpha_2(y_i) y_i^r:~\sum_{i=1}^k y_i\geq n,\,y_i\in[0,n]  \}}\int_{\R^k_+} \prod_{i=1}^k e^{-\frac{c}{4}\alpha_2(y_i) y_i^r} {\rm d} \mbox{\boldmath $y$}\\
&\leq C_{c,\rho,k}\, \exp{\left(-\left(1-\frac{c}{2}\right) \inf\left\{\sum_{i=1}^k \alpha_2(y_i) y_i^r:~\sum_{i=1}^k y_i\geq n  ,\,y_i\in[0,n]\right\}\right)},
}
where $C_{c,\rho,k}>0$ is a constant independent of $n$. Let $(y_i)_{i=1}^k\in [0,n]^k$ be such that $\sum_{i=1}^k y_i\geq n$. Then, by Lemma~\ref{local unif SV}, for $n$ large enough,  we get
\al{
\sum_{i=1}^k \alpha_2(y_i) y_i^r&\geq \sum_{i=1}^k \alpha_2(y_i) y_i^r {\mathbf{1}_{y_i\geq \left(\frac{c}{4k}\right)^{1/r} n}}\\
&\geq \left(1-\frac{c}{4}\right) \alpha_2(n) \sum_{i=1}^k y_i^r {\mathbf{1}_{y_i\geq \left(\frac{c}{4k}\right)^{1/r} n}}\\
&\geq \left(1-\frac{c}{4}\right) \alpha_2(n) \left( \sum_{i=1}^k y_i^r - \frac{cn^r}{4}\right)\geq \left(1-\frac{c}{2}\right) \alpha_2(n) n^r,
}
where we have used \eqref{concave} in the last line. Putting things together, we have
\al{
\P\left(\sum_{i=1}^k Y_i>n\right)\leq  e^{-(1-\frac{2c}{3}) \alpha_2(n) n^r}+ C_{c,\rho,k} e^{-\left(1-\frac{c}{2}\right)^2 \alpha_2(n) n^r}\leq e^{-(1-c) \alpha_2(n) n^r}.
}
\end{proof}
Lemma~\ref{lem:Zhang} is replaced by the following. 
\begin{lem}
For any $\e>0$, there exist $M=M(\e)\in\N$ and a positive constant $c=c(\e)$ such that for $n\in\N$ large enough,
$$\P\left({\rm T}_{\R\times[-M,M]^{d-1}}\left(0,n\mathbf{e}_1\right)\geq (\mu+\e)n\right)\leq e^{-c { \alpha_2}(n) n^r}.$$
\end{lem}
Using Lemma~\ref{exp:est2} instead of Lemma~\ref{exp:est}, the proof is the same as in Lemma~\ref{lem:Zhang}, so we omit this. 
\section{Proof of (\ref{thm:main4})}\label{slow-vary2}
We first consider the lower bound. We follow the argument in Section~\ref{section: lower bound r>1}, though the side lengths of the boxes are different. Let $\e>0$ arbitrary. We take $M>0$ such that $\lambda_{d,r}(M)\leq (1+\e) \lambda_{d,r}$. Let
 $${\rm T}^{[1]}_n = {\rm T}_{D_{M}(0)}(0,\partial D_{M}(0))\text{ and }{\rm T}^{[2]}_n = {\rm T}_{D_{M}(n \mathbf{e}_1)}(n\mathbf{e}_1,\partial D_{M}(n \mathbf{e}_1)).$$  We consider the following events:

\[F_1 = \left\{{\rm T}_n^{[1]} \geq \frac{(\xi+\varepsilon) n}{2}\right\}, \quad F_2 = \left\{ {\rm T}_n^{[2]} \geq \frac{(\xi+\varepsilon) n}{2}\right\},\]
and
\[G = \left\{\min_{x\in \partial D_{M}(0),y\in \partial D_{M}(n \mathbf{e}_1) } {\rm T}_{D_{M}(0)^c \cap D_{M}(n \mathbf{e}_1)^c}(x,y) \geq (\mu - \varepsilon) n \right\}.\]
 As \eqref{eq:lowerBoundFG}, we have
\begin{equation*}
\P\left({\rm T}_n\geq (\mu+\xi) n \right) \geq \P(F_1)^2 \,\P(G).
\end{equation*} 

We begin with estimating $\P(F_1)$. 
Let $(t_e^\star)_{e\in E_{M}}$ that minimizes $\lambda^P_{d,r}(D_{M}(0),\partial D_{M}(0))$. In particular,
\[
\left\{\forall e\in E_{M},  \tau_e \geq \frac{(\xi + \varepsilon)n}{2} t_e^\star\right\} \subset  \left\{{\rm T}_n^{[1]} \geq \frac{(\xi+\varepsilon) n}{2}\right\} = F_1.\]
On the other hand,

\begin{align*}
 & \qquad \P\left(\forall e\in E_{M},  \tau_e  \geq \frac{(\xi + \varepsilon)n}{2} t_e^\star\right) \\
&=  \prod_{e\in E_{M}}\P\left( \tau_e  \geq \frac{(\xi + \varepsilon)n}{2} t_e^\star\right) \\
  &\geq \left(\prod_{e:~t_e^\star>0}\,\beta_1\left(\frac{(\xi+\varepsilon)n}{2}t_e^\star\right)\right) \exp{\left(- \left(\frac{(\xi + \varepsilon)n}{2}\right)^r \sum_{e\in E_{M}} \alpha_1\left(\frac{(\xi + \varepsilon)t_e^\star n}{2}\right) (t_e^\star)^r\right)}\\
  & \geq \exp{\left(-(1+\e)\alpha_1(n) \left(\frac{(\xi + \varepsilon)n}{2}\right)^r \lambda_{d,r}(M)\right)}\\
  &\geq \exp{\left(-(1+\e)^2\alpha_1(n) \left(\frac{(\xi + \varepsilon)n}{2}\right)^r \lambda_{d,r}\right)},
\end{align*}
where we have used \eqref{def:lambdaIntro} and Lemma~\ref{local unif SV}. This yields
\begin{equation*}
\varliminf_{n\to\infty} \frac{1}{\alpha_1(n)\,n^r} \log \P(F_1) \geq - (1+\e)^2 \left(\frac{\xi + \varepsilon}{2}\right)^r  \lambda_{d,r}. 
\end{equation*}
As in Section~\ref{section: lower bound r>1}, we obtain $\lim_{n\to\infty} \P(G) = 1$. Putting things together,  letting $n\to\infty$ followed by $\varepsilon\to 0$,  we obtain:
\begin{equation*}
 \varliminf_{n\to\infty} \frac{1}{\alpha_1(n) \, n^r} \log \P\left({\rm T}(0,nx) \geq (\mu+\xi) n \right)\geq -  2^{1-r} \xi^r \lambda_{d,r}.
 \end{equation*}

Next, we consider the upper bound. We write $D_M=D_M(0)$ and $\partial D_M=\partial D_M(0)$. We follow the arguments in Section~\ref{subsec:proofUpperBound} and use the same notations except for $\ell_M(n)$, where we use $\ell^{\chi}(n)=n^\chi$ with a fixed constant $\chi\in \left(\frac{r-1}{d-1},\frac{r}{d} \right)$ instead of $\ell_M(n)$. We note that by Lemma~\ref{lem: Karamata}, for any $p<r$, there exist $\alpha,\beta>0$ such that 
$$\P(\tau_e\geq t)\leq \beta e^{-\alpha t^{p}}.$$
We can apply Lemma~\ref{lem:Zhang} and Theorem~\ref{prop1} with $p$ close enough to $r$ so that $p \chi >r$ to estimate $A_2^{[1]},A_2^{[2]},B_2$. In fact, similarly to \eqref{eq:sumIndependent}, we obtain that for any $L>0$ and $n$ large enough
\begin{equation*}
\begin{aligned}
  & \P({\rm T}_n  > (\mu + \xi)n) \\
  &\leq e^{-L\alpha_2(n) n^r}+\sum_{k=0}^{\lf(\xi - 3\varepsilon) n\rf} \P\left( {\rm T}^{[1]}_n\geq k \right) \P({\rm T}^{[2]}_n  \geq (\xi - 3\varepsilon) n-k-1),
\end{aligned}
\end{equation*}
where 
\[{\rm T}^{[1]}_n:={\rm T}_{D_{2\ell^{\chi}(n)}(0)}\left(0,I_{2\ell^{\chi}(n)}^{[1]}\right),\] 
\[{\rm T}^{[2]}_n:={\rm T}_{D_{2\ell^{\chi}(n)}(n\mathbf{e}_1)}\left(n\mathbf{e}_1,I_{2\ell^{\chi}(n)}^{[2]}\right).\] 
Given $\e,\rho>0$, we choose a random variable $X$ that has a continuous density $c_{\e,\rho} e^{-(1-\frac \e 2) \alpha_2(t) t^r}\mathbf{1}_{\{t\geq \rho\}}$ with some $c_{\e,\rho}>0$. Then, for $\rho>0$ large enough, $\tau_e$ is stochastically dominated by $X$  as in the proof of Lemma~\ref{exp:est2}. We further suppose that for any $t>\rho$, the Karamata representation \eqref{Karamata repre} holds for $\alpha_2(t)$. Similarly to \eqref{eq:boundLaplace}, since $X\geq \rho$ a.s., we obtain that
\begin{equation*} 
\qquad \P\left( {\rm T}^{[1]}_{n}    \geq k \right)
   \leq  e^{c_{\varepsilon,d} \ell^{\chi}(n)^d} e^{-(1-\varepsilon) k^r  {\lambda}^{\rho,k}_{d,r}(2\ell^{\chi}(n))},
\end{equation*}
where we define
\al{
&\qquad \lambda_{d,r}^{\rho,k}(M)\\
&=\inf_{(t_e)\in [\frac{\rho}{k},\infty)^{E_M}} \left\{\sum_{e\in E_M} \alpha_2(k t_e) \,t_e^{r}\ \middle|~\forall \gamma:0\to \partial D_M\text{ with $\gamma\subset D_M$},\,\sum_{e\in\gamma}t_e\geq 1\right\}.
}
 Lemma~\ref{lem: Karamata} shows that  for any $\e\in(0,1)$, if $k$ is sufficiently large, then  
\ben{\label{SV prop2}
\frac{\alpha_2(k)}{\alpha_2(k t)}\leq t^{r},\,\forall t\in[1+\e,\infty),\,\frac{\alpha_2(k t)}{\alpha_2(k)}\geq (1-\e)t^{\e},\,\forall t\in\left[\frac{\rho}{k},1+\e\right].
}
The first inequality  shows that $\alpha_2(k)\leq \alpha_2(k t) t^r$ for $t>1+\e$ and thus, for $k$ large enough, if $(t_e^\star)$ is the minimizer of $\lambda_{d,r}^{\alpha,k}(2\ell^{\chi}(n))$, then $t_e^\star\leq 1+\e$ for any $e\in E_{2\ell^{\chi}(n)}$. Hence, by the second inequality of \eqref{SV prop2}, we have  $$\lambda_{d,r}^{\rho,k}(2\ell^{\chi}(n))\geq (1-\e) \alpha_2(k) \lambda_{d,r+\e}(2\ell^{\chi}(n))\geq (1-\e) \alpha_2(k) \lambda_{d,r+\e},$$ for $k$ large enough and we obtain
\al{
\P\left( {\rm T}^{[1]}_{n}    \geq k \right)
   \leq  \exp{\left(c_{\varepsilon,d} \ell^{\chi}(n)^d-\alpha_2(k) (1-\varepsilon)^2 k^r  {\lambda}_{d,r+\e}\right)}.
}
The same estimate holds for ${{\rm T}}^{[2]}_n$. We note that either $k$ or $(\xi - 3\varepsilon) n-k-1$ is large.
Therefore, we get
\begin{align*}
&\qquad \sum_{k=0}^{\lf(\xi - 3\varepsilon) n\rf} \P\left( {\rm T}^{[1]}_n\geq k \right) \P({\rm T}^{[2]}_n  \geq (\xi - 3\varepsilon) n-k-1)\\
& \leq e^{2c_{\varepsilon,d} M^d \ell^{\chi}(n)^d} \sum_{k=0}^{\lf(\xi - 3\varepsilon) n\rf}        
 e^{-(1-2\varepsilon)\left( \alpha_2(k) k^r +\alpha_2((\xi - 3\varepsilon)n-k) ((\xi - 3\varepsilon)n-k)^r\right) \lambda_{d,r+\e}} \nonumber\\
& \leq \xi n \, \exp{\left(2c_{\varepsilon,d} M^d \ell^{\chi}(n)^d-2\alpha_2(n) (1-3\varepsilon)(\frac{\xi - 3\varepsilon}{2})^r n^r \lambda_{d,r+\e}\right)}, 
\end{align*}
where we have used $(s\xi)^r + ((1-s)\xi)^r\geq 2 (\xi/2)^r$ and  Lemma~\ref{local unif SV} in the last line.  Together with $\ell^{\chi}(n)^d\ll \alpha_2(n) n^r$, $\lim_{\e\to 0}\lambda_{d,r+\e}=\lambda_{d,r}$,  letting $n\to\infty$ followed by $\varepsilon\to 0$,  we obtain:
\begin{equation*} 
 \varlimsup_{n\to\infty} \frac{1}{\alpha_2(n) \, n^r} \log \P\left({\rm T}(0,nx) \geq (\mu+\xi) n \right)\geq -  2^{1-r} \xi^r \lambda_{d,r}.
 \end{equation*}

\section*{Acknowledgements}
 We thank Ofer Zeitouni for suggesting Theorem \ref{th:rLessThandBig1Loc}. {We thank the referee for his valuable comments}. We also thank Bobo Hua and Florian Schweiger for helpful discussions. The second author would like to thank Ryoki Fukushima introducing \cite[Lemma~3.1]{CZ03}. The first author acknowledges that this project has received funding from the European Research
Council (ERC) under the European Union Horizon 2020 research and innovation program
(grant agreement No. 692452).  The second author is partially supported by  JSPS KAKENHI 19J00660 and SNSF grant 176918.

\end{document}